\documentclass[a4paper]{amsart}

\usepackage{amsthm,amssymb}
\usepackage{mathtools,thmtools}
\usepackage{bm,esint}
\usepackage{color}
\usepackage{float,graphicx,subcaption}
\usepackage{hyperref}
\usepackage{cancel}

\title{
  On concentration in vortex sheets
}

\author{Samuel~Lanthaler}
\address[Samuel Lanthaler]{California Institute of Technology, Pasadena, CA 91125, USA}

\thanks{\textbf{Acknowledgements:} 
  The author would like to thank Prof. Siddhartha Mishra for valuable discussions. This research was partially supported by the European Research council (ERC) consolidator grant ERC COG 770880: COMANFLO}

\newcommand{\explain}[2]{\overset{\mathclap{\underset{\downarrow}{#2}}}{#1}}
\newcommand{\loc}{{\mathrm{loc}}}
\newcommand{\dirac}{\delta}
\newcommand{\slot}{{\,\cdot\,}}
\newcommand{\dom}{{\R^2}}
\newcommand{\supp}{\mathrm{supp}}

\newcommand{\weaklyto}{{\rightharpoonup}}
\newcommand{\weakstarto}{\overset{\ast}{\weaklyto}}

\newcommand{\R}{\mathbb{R}}

\renewcommand{\div}{{\mathrm{div}}}
\newcommand{\curl}{{\mathrm{curl}}}
\newcommand{\Lip}{\mathrm{Lip}}

\newcommand{\D}{\mathcal{D}}

\newcommand{\define}{\textbf}

\declaretheoremstyle[
  headfont=\normalfont\bfseries\itshape,
  numbered=unless unique,
  bodyfont=\normalfont,
  spaceabove=1em plus 0.75em minus 0.25em,
  spacebelow=1em plus 0.75em minus 0.25em,
  qed={},
]{deflt}

\theoremstyle{deflt}
\newtheorem{theorem}{Theorem}[section]

\newtheorem{remark}[theorem]{Remark}

\newtheorem{definition}[theorem]{Definition}
\newtheorem{lemma}[theorem]{Lemma}
\newtheorem{proposition}[theorem]{Proposition}
\newtheorem{corollary}[theorem]{Corollary}

\numberwithin{equation}{section}
\numberwithin{theorem}{section}

\begin{document}

\begin{abstract}
The question of energy concentration in approximate solution sequences $u^\epsilon$, as $\epsilon \to 0$, of the two-dimensional incompressible Euler equations with vortex-sheet initial data is revisited. Building on a novel identity for the structure function in terms of vorticity, the vorticity maximal function is proposed as a quantitative tool to detect concentration effects in approximate solution sequences. This tool is applied to numerical experiments based on the vortex-blob method, where vortex sheet initial data without distinguished sign are considered, as introduced in \emph{[R.~Krasny, J. Fluid Mech. \textbf{167}:65-93 (1986)]}. Numerical evidence suggests that no energy concentration appears in the limit of zero blob-regularization $\epsilon \to 0$, for the considered initial data.
\\[1em]
\end{abstract}

\maketitle

\section{Introduction} \label{sec:intro}

The present work considers the question of concentration in solutions of the two-dimensional incompressible Euler equations with vortex sheet initial data. As is well-known (see e.g. \cite{MajdaBertozzi} and references therein), the incompressible Euler equations can be formulated either in the primitive variable formulation, providing an evolution equation for the fluid velocity $u$, or in the vorticity formulation, which in the two-dimensional case describes the advection of the vorticity $\omega = \curl(u)$ by the velocity field ${u}$. More precisely, the primitive variable formulation of the two-dimensional incompressible Euler equations is given by the following evolution equation for $u: \,\R^2 \times [0,T] \to \R^2$, $(x,t) \mapsto u(x,t)$:
\begin{gather} \label{eq:Euler}
\left\{
\begin{aligned}
\partial_t u + \div(u\otimes u) + \nabla p 
&= 0, \\
\div(u) 
= 0, \quad
u(t=0) 
&= u_0.
\end{aligned}
\right.
\end{gather}
Here, $u_0$ is the initial data and $p$ is the (scalar) pressure which can be interpreted as a Lagrange multiplier enforcing the incompressibility constraint $\div(u) = 0$. Taking the curl of \eqref{eq:Euler}, the following vorticity equation is formally obtained for $\omega = \curl(u)$ \cite{MajdaBertozzi}
\begin{gather} \label{eq:vorticity}
\left\{
\begin{gathered}
\partial_t \omega + \div(\omega \, u) 
= 0, \\
\omega(t=0) 
= \omega_0.
\end{gathered}
\right.
\end{gather}

\subsection{Theoretical results}
While existence and uniqueness results are available for sufficiently smooth initial data, e.g. if $\omega_0 \in L^\infty\cap L^1$ \cite{Yudovich1995} (and closely related spaces \cite{Yudovich1963,Chemin1996,Vishik1999}) or $\omega_0 \in L^\infty$, $u_0\in L^\infty$ \cite{Serfati1995,AKLN2015}, for less regular initial data of physical interest such as vortex sheet initial data, many open questions remain. In particular, it is currently not known whether weak solutions exist for vorticity initial data $\omega_0 \in \mathcal{M} \cap H^{-1}$ a bounded Radon measure in the Sobolev space $H^{-1}$. Existence has been proved for $\omega_0 \in L^p$, $p>1$ in the pioneering work \cite{DipernaMajda1987a} and more generally for rearrangement invariant spaces with compact embedding in $H^{-1}_{\mathrm{loc}}$ \cite[Section 4.2]{Lions}, \cite{LNT2000}. For $\omega_0 \in \mathcal{M}$ a finite Radon measure, existence is so far only known under a sign restriction. Existence has in this case been obtained in the celebrated work of Delort \cite{Delort1991} (see also \cite{Majda1993,Schochet1995}). In the work \cite{Delort1991}, the incompressible Euler equations \eqref{eq:Euler} are interpreted in the weak formulation: $u \in L^\infty(0,T;L^2_x)$ is a weak solution of \eqref{eq:Euler} with initial data $u_0 \in L^2_x$, if 
\begin{gather}\label{eq:Eulerweak}
\int_0^T \int_\dom u \cdot \partial_t \phi + \nabla \phi : (u\otimes u) \, dx \, dt
= -\int_\dom u_0 \cdot \phi(t=0) \, dx,
\end{gather}
for all $\phi \in C^\infty_c([0,T)\times \dom;\R^2)$, such that $\div(\phi) = 0$. Delort then observed that given an approximate solution sequence $u_n \weaklyto u$ converging weakly to $u\in L^2_tL^2_x$, it is possible (by a \emph{compensated-compactness} argument) to pass to the limit in the non-linear term
\[
\int_0^T \int \nabla \phi: (u_n\otimes u_n) \, dx \, dt
\to
\int_0^T \int \nabla \phi: (u\otimes u) \, dx \, dt,
\]
provided that the vorticity is uniformly bounded $\Vert \omega_n(t) \Vert_{\mathcal{M}}, \Vert \omega_n(t) \Vert_{H^{-1}}\le C$ for all $t\in [0,T]$, $n\in \mathbb{N}$, and that the following non-concentration property is satisfied
\[
\lim_{r\to 0} \sup_{n} \sup_{t\in[0,T]} M_r(\omega_n(t)) = 0.
\]
Here $M_r(\omega(t))$ is the ``vorticity maximal function'' \cite{DipernaMajda1987a,DipernaMajda1988}, defined by
\begin{align} \label{eq:vortmax}
M_r(\omega) := \sup_{x \in \dom} \int_{B_r(x)} |\omega(y,t)| \, dy.
\end{align}
For essentially \emph{non-negative} vorticity $\omega_0 = \omega_0' + \omega_0''$ with $\omega_0' \in H^{-1}\cap \mathcal{M}_+$, $\omega_0' \ge 0$ a non-negative finite measure, and $\omega_0'' \in L^p\cap H^{-1}$, it is shown in \cite{Delort1991} (for $p>1$; see \cite{VecchiWu1993} for $p=1$), that the approximate solution sequence $\omega_n$ obtained by mollification of the initial data satisfies the non-concentration property. The existence of weak solutions with initial data in this ``Delort-class'' follows by Delort's compensated-compactness argument. In fact, it has been shown in \cite{Majda1993} that if $\omega_0 \in \mathcal{M}_+ \cap H^{-1}$ is a non-negative measure such that $|x|^2 \omega_0(x) \in \mathcal{M}_+$ is finite, then there exists a solution of the incompressible Euler equations with vorticity maximal function satisfying the following a priori bound: There exists a constant $C>0$, such that
\begin{align} \label{eq:apriorivortmax}
M_r(\omega(t)) \le C|\log(r)|^{-1/2}, \quad \text{for } r\le \frac 12,
\end{align}
for almost all $t\in [0,T]$. An extension of Delort's result \cite{Delort1991} to vortex sheet data with reflection symmetry and a sign restriction in a half-space has been achieved in \cite{LNZ2001}.

For vortex sheet initial data $\omega \in \mathcal{M} \cap H^{-1}$ without a sign-restriction, no general (unconditional) existence results are at present available. In particular, it is not known whether weakly convergent approximate solution sequences $u^\epsilon \weaklyto u$, $\epsilon > 0$, obtained either by solution of the incompressible Euler equations with regularized initial data or from regularized equations of motion may exhibit concentration phenomena in the limit of zero regularization $\epsilon \to 0$ \cite{DipernaMajda1987a,DipernaMajda1987b}. The study of possible concentration and oscillation phenomena in approximate solution sequences of the incompressible Euler equations was pioneered by Diperna and Majda in a series of papers \cite{DipernaMajda1987a,DipernaMajda1987b,DipernaMajda1988}. In \cite{DipernaMajda1987a,DipernaMajda1987b}, the concept of measure-valued solutions of the incompressible Euler equations is introduced, taking into account oscillation and concentration phenomena in the limit $\epsilon \to 0$. For vortex-sheet initial data with $\omega_0 \in \mathcal{M}$, it can be shown that suitable regularizations do not exhibit oscillations. Concentrations in the limit $\epsilon \to 0$ might occur, and have been investigated further in \cite{DipernaMajda1988}. In \cite{DipernaMajda1988} (see also \cite{GreengardThomann1988,Zheng1991}), it is shown that a form of compensated compactness (or ``concentration-cancellation'') would follow from bounds on the Hausdorff-dimension of the concentration set. Such a concentration-cancellation result has been established for the case of \emph{time-invariant} approximate solution sequences in \cite{DipernaMajda1988}. Whether concentration-cancellation occurs more generally for approximate solution sequences is still not known, however.

\subsection{Numerical approximation}
In the absence of general existence and uniqueness results for vortex sheet initial data, several previous investigations have therefore resorted to elucidate the complicated dynamics of vortex sheets by numerical methods.  
Popular numerical approximation schemes include grid-based (Eulerian) methods such as finite difference/finite volume \cite{LevyTadmor1997} and spectral methods \cite{LM2015,LM2019}, as well as Lagrangian methods, such as the vortex-point \cite{Rosenhead1931,krasny1986study} and vortex-blob methods \cite{Chorin1973}. Convergence results to solutions of the incompressible Euler equations have been obtained for rough initial data with $\omega_0$ belonging to a rearrangement invariant space with compact embedding in $H^{-1}_{\mathrm{loc}}$ \cite{LNT2000} and, more recently, for periodic initial data belonging to the Delort-class $\omega_0 \in H^{-1}\cap ( \mathcal{M}_+ + L^1)$ in \cite{LM2019}. For vortex-point and vortex-blob methods, convergence for initial data $\omega_0 \in H^{-1} \cap \mathcal{M}_+$ are also available \cite{LiuXin1995,Schochet1996,LiuXin2001}. In a series of numerical experiments presented by Krasny \cite{Krasny1986,Krasny1987,Krasny1990}, the evolution of vortex sheets by a regularised Birkhoff-Rott equation, using the vortex-blob method, have been illustrated. In these numerical experiments, a marked difference has been observed between the dynamics of vortex sheets with and without a sign-restriction on the vorticity \cite{Krasny1987}; the unsigned case apparently exhibiting considerably more complicated dynamics compared to the signed case. This observation has prompted the conjecture \cite[Remark 3.2]{DipernaMajda1987a}, \cite[p.447]{MajdaBertozzi} that the observed small-scale features in the unsigned case might hint at a convergence to a non-trivial measure-valued solution (with concentration in the limit) \cite{DipernaMajda1987a,DipernaMajda1987b}. 

In contrast, recent numerical experiments based on the spectral viscosity method in the context of statistical solutions to the Euler equations \cite{LMP2019,LMP2020} have found no apparent qualitative differences in the regularity of numerical approximations to vortex sheet initial data with and without a sign restriction (at least for ``typical'' initial data among randomly perturbed vortex sheets). In \cite{LMP2019,LMP2020}, the regularity of the numerically obtained approximate solutions was measured by the decay of the structure functions $S_2(u^\Delta;r)$ of the discretized velocity $u^\Delta$ obtained at grid scale $\Delta > 0$:
\begin{align} \label{eq:structfun}
S_2(u^\Delta(t);r) := \left(\fint_{B_r(0)}\int_{D} |u^\Delta(x+h,t) - u^\Delta(x,t)|^2 \, dx \, dh \right)^{1/2}.
\end{align}
One of the main findings of that work was an apparent uniform (for $t\in [0,T]$ and as $\Delta \to 0$) decay of these structure functions. More precisely, the numerical experiments indicate that there exist constants $C,\alpha > 0$, such that $\sup_{t\in [0,T]} S_2(u^\Delta(t);r) \le C r^\alpha$, for all resolutions $\Delta > 0$. As explained in \cite{LMP2019}, such a uniform decay implies strong $L^2$-compactness of the approximate solution sequence $u^\Delta$, and energy-conservation of any solution $u$ obtained in the limit $\Delta \to 0$. 

\subsection{Contributions and scope of the present work} 

The complicated dynamics of unsigned vortex sheets show a creation of intricate small-scale structures which results from the interaction of parts with positive and negative vorticity, and which appears to be absent in cases with a distinguished sign. This might indicate that the evolution of vortex sheets without a distinguished sign exhibit much less regularity than the evolution of vortex sheets with vorticity of distinguished sign $\omega_0 \ge 0$ \cite[Remark 3.2]{DipernaMajda1987a}. However, beyond the visually very compelling qualitative differences in the evolution of vortex sheets with and without a distinguished sign uncovered in \cite{Krasny1987}, a detailed \emph{quantitative analysis} of these results appears not to have been carried out so far. Recent numerical experiment based on the spectral viscosity method \cite{LMP2019,LMP2020} have found no essential differences, when comparing the regularity at later times in the evolution of approximate solutions with signed and unsigned vortex sheet initial data. This apparent discord between the expected deterioration of regularity from results based on the vortex-blob method and the observed persistence of regularity in numerical experiments conducted with spectral methods provides the main motivation for the present work. The numerical experiments considered by Krasny \cite{Krasny1987} are revisited, and the regularity of the resulting approximate solutions is analysed \emph{quantitatively}. The main contributions are as follows:
\begin{itemize}
\item  A novel identity for the structure function \eqref{eq:structfun} is derived in terms of the vorticity. From this identity, an explicit estimate on the decay of the structure functions is obtained in terms of the vorticity maximal function. In particular, it is shown that an algebraic decay of the vorticity maximal function $M_r(\omega)$ implies an algebraic decay of the structure function $S_2(u;r)$.
\item As an immediate corollary of the novel identity, we recover the logarithmic circulation theorem of Diperna and Majda \cite[Theorem 3.1]{DipernaMajda1987b}, under slightly relaxed assumptions.
\item This relation between the vorticity and the associated flow represents the main tool to analyse our numerical experiments. We prove strong convergence, to a energy-conservative solution of the incompressible Euler equations, of approximate solution sequences obtained from the vortex-blob method, under the assumption of an algebraic decay of the vorticity maximal function.
\item Based on these theoretical considerations, numerical experiments  employing the vortex-blob method are presented and analysed. It is shown that despite the -- visually -- intricate and complex evolution of vortex sheets without a distinguished sign, the quantitative analysis based on the vorticity maximal function shows \textbf{no indication of concentration} in the evolution of vortex sheets, even as these complex vortex sheet dynamics take place.
\end{itemize}

\subsection{Organisation}
In Section \ref{sec:structfun}, we derive a novel exact expression for the structure function $S_2(u;r)$ in terms of the vorticity, and establish the link to the vorticity maximal function. In Section \ref{sec:vortexblob}, we review the formulation of the vortex-blob method, and based on the results of Section \ref{sec:structfun}, we give a sufficient condition for the strong convergence of approximate solutions obtained from the vortex-blob method to a solution of the incompressible Euler equations. In Section \ref{sec:numerical}, we revisit the numerical experiments conducted in \cite{Krasny1987}.  Conclusions and perspectives for future work are given in Section \ref{sec:conclusion}.

\subsection{Notation} We will denote spaces such as $L^2(\dom;\R^2)$ consisting of two-dimensional vector fields with spatial dependency $x\mapsto u(x)$ by $L^2_x$. Similar notation will also be used for functions, i.e. $L^1_x$ may refer to $L^1(\dom)$ (with target space $\R$), if the target space is clear from the context. Similarly, for fixed $T>0$, Bochner spaces of time-dependent vector fields, such as $L^p(0,T;L^2_x)$, consisting of measurable $(x,t) \mapsto u(x,t)$ with $(x,t) \in \dom\times [0,T]$, such that
\[
\Vert u \Vert_{L^p_tL^2_x} = \left(\int_0^T \Vert u \Vert_{L^2_x}^p \, dt \right)^{1/p} < \infty,
\]
will be denoted more simply by $L^p_tL^2_x$. The symbol $\mathcal{M}$ refers to the space of bounded Radon measures on $\dom$, throughout this work. Other (mostly standard) notation is introduced as needed, and should be clear from the context. We will follow the convention that in the derivation of estimates, a single letter $C$ will be used throughout, even if the value of the constant differs from line to line. The final dependency of the constant on other given quantities will be made precise in each case.

\section{Structure function decay and vorticity correlations} \label{sec:structfun}

We first recall that approximate solutions obtained by mollifying the initial data, by solving regularized equations such as the Navier-Stokes equations or obtained from most numerical schemes produce approximate solutions sequences in the limit of vanishing regularisation according to the following definition, similar to \cite{DipernaMajda1987a}:
\begin{definition} \label{def:approxsol}
A sequence of vector fields $u^\epsilon \in C([0,T];L^2_x)$, $\epsilon \to 0$, with vorticity $\omega^\epsilon(t) = \curl(u^\epsilon(t))\in \mathcal{M}$  is a \define{approximate solution sequence} of the incompressible Euler equations with initial data $u_0\in L^2_x$ and $\omega_0 = \curl(u_0) \in \mathcal{M}$  on the time interval $[0,T]$, if the following four conditions are satisfied:
\begin{enumerate}
\item $\sup_{t\in [0,T]} \Vert u^\epsilon(t) \Vert_{L^2} \le C$, uniformly for all $\epsilon>0$,
\item $\sup_{t\in [0,T]} \Vert \omega^\epsilon(t) \Vert_{\mathcal{M}} \le C$, uniformly for all $\epsilon>0$,
\item weak consistency with the Euler equations \eqref{eq:Eulerweak}: For some $L>0$, $u^\epsilon \in \mathrm{Lip}([0,T];H^{-L})$ is uniformly Lipschitz continuous with values in the negative Sobolev space $H^{-L}$ and
\[
\lim_{\epsilon \to 0}
\int_0^T \int_\dom ( \partial_t \phi) \cdot u^\epsilon + \nabla \phi: (u^\epsilon\otimes u^\epsilon) \, dx \, dt = 0,
\]
for all $\phi \in C_c^\infty(\dom \times (0,T))$ with $\div(\phi) = 0$.
\item convergence to initial data: We have $u^\epsilon(t=0) \to u_0$ strongly in $L^2$ and $\omega^\epsilon(t=0) \weakstarto \omega_0$ weak-$\ast$ with respect to the natural pairing of $\mathcal{M}$ with $C_0(\dom)$.
\end{enumerate}
\end{definition}

\begin{remark}
The vortex sheets considered in the numerical experiments conducted in this work have vanishing mean vorticity, $\int_\dom \omega^\epsilon \, dx = 0$ and uniformly bounded velocity $u$ (in the $L^2$-norm). This motivates assumption (1) in our formulation of Definition \ref{def:approxsol}. If cases with $\int_\dom \omega^\epsilon \, dx \ne 0$ were to be considered, a more careful discussion of suitable decay conditions for the velocity at $x \to \infty$ would be required (see \cite[Remark 1.1]{DipernaMajda1987a}). 
\end{remark}

The structure functions $S_2(u^\epsilon;r)$ in the form \eqref{eq:structfun} have originally been introduced as a measure of spatial correlations in approximate solution sequences and statistical correlations in \cite{FLM2017}. In different variations, such quantities have also extensively been discussed in the turbulence literature. In the deterministic case, these structure functions can be interpreted as a precise measure of the strong $L^2$-compactness of approximate solution sequences. This has been shown for periodic flows in \cite{LMP2020}. In the present non-periodic setting, we state the following proposition, which is a variant of \cite[Proposition 2.7]{LMP2020}:
\begin{proposition} \label{prop:compactness}
Let $\{u^\epsilon\}_{\epsilon > 0}$, be an approximate solution sequence of the incompressible Euler equations. If there exists a uniform modulus of continuity $\phi(r)$, such that 
\[
\sup_{t\in [0,T]} S_2(u^\epsilon(t); r) \le \phi(r), \quad \forall r > 0,
\]
uniformly in $\epsilon >0$, then $\{u^\epsilon\}_{\epsilon > 0}$ is strongly precompact in $L^2_x([0,T];L^2_{x,\mathrm{loc}})$.
\end{proposition}

\begin{remark}
The restriction to local compactness in $L^2_{x,\loc}$, rather than global compactness in $L^2_x$, is due to the fact that the present assumptions cannot rule out that ``mass leaks out at infinity''. A simple example is provided by fixing any non-trivial $w\in L^2_x$, and defining $u^\epsilon(x) := w(x + x_0/\epsilon)$ for some $x_0\ne 0$. This sequence converges locally $u^\epsilon \overset{\loc}{\to} 0$, but doesn't converge globally, $\Vert u^\epsilon \Vert_{L^2_x} \equiv \Vert w \Vert_{L^2_x} \not\to 0$.
\end{remark}

\begin{proof}[Sketch of proof]
The proof is an almost verbatim repetition of the proof of the ``only if'' direction in \cite[Proposition 2.7]{LMP2020}, and the argument is not repeated here in detail: The statement is essentially an application of Kolmogorov's characterization of compact subsets of $L^p$ spaces, or more precisely the Bochner space analogue of this characterization presented in \cite[Section 3, Theorem 1]{Simon1986}. The only non-trivial ingredient is to establish uniform continuity in time, as is required for application of this theorem. To prove uniform continuity in time, one observes that $u^\epsilon$ is an approximate solution sequence, and therefore possess some minimal uniform continuity $u^\epsilon \in \mathrm{Lip}([0,T];H^{-L})$ in time, albeit with very weak spatial regularity. It turns out that the present assumption of a uniform bound on the structure functions $S_2(u^\epsilon(t);r)$ (providing spatial regularity), can be leveraged to show that $t\mapsto u^\epsilon(t)$ possesses temporal uniform continuity in $L^2_t(0,T;L^2_{x,\loc})$.
\end{proof}

\begin{remark}
 As described in the sketch of the proof above, Proposition \ref{prop:compactness} essentially states that any sequence which has ``good temporal continuity with very weak spatial regularity'' (e.g. uniform bound in $\Lip_tH^{-L}_x$) and,  at the same time, ``very weak temporal continuity with good spatial regularity'' (e.g. uniformly bounded structure functions, analogous to uniform bound in $L^2_t H^\alpha_x$, $\alpha > 0$)  must necessarily also have ``some uniform temporal continuity with some uniform spatial regularity'' (compactness in $L^2_t L^2_x$).
\end{remark}

\begin{remark} \label{rem:concentration}
We recall that for any approximate solution sequence $u^\epsilon \in L^2$, $\epsilon \to 0$, there exists a (subsequential) weak limit $u^\epsilon \weaklyto u$ in $L^2_t L^2_{x,\mathrm{loc}}$. As explained in \cite{DipernaMajda1987a,DipernaMajda1987b}, the lack of strong convergence is captured by the concentration measure: Under the assumption of uniform control $\Vert \omega^\epsilon \Vert_{\mathcal{M}}$, it can be shown that there exists a time-parametrized, weak-$\ast$ continuous measure $t\mapsto \lambda_t \in \mathcal{M}$, such that $\lambda_t(dx) \perp dx$ is supported on a set of Lebesgue measure zero, for all $t\in [0,T]$, and (again, up to the extraction of a subsequence)
\[
\int_0^T \int_\dom \phi |u^\epsilon|^2 \, dx \, dt
\to 
\int_0^T \int_\dom \phi |u|^2 \, dx \, dt
+ 
\int_0^T \int_\dom \phi \, d\lambda_t(x) \, dt,
\]
for any $\phi \in C_c(\dom\times [0,T])$. The concentration measure $\lambda_t$ thus characterizes energy concentration on sets of measure zero and the lack of strong convergence in approximate solution sequences. In particular, a uniform bound on the structure functions as in the assumptions of Proposition \ref{prop:compactness} implies a lack of energy concentration.
\end{remark}

\begin{remark} \label{rem:kernel}
\emph{(cp. Chapter 3 of \cite{MajdaBertozzi})} If the vorticity $\omega \in L^1_x\cap L^\infty_x$, then a velocity field $u \in L^2_{x,\loc}$ such that $\div(u) = 0$, $\curl(u) = \omega$ is given by
\begin{align} \label{eq:kernelrep}
u(x) = \int_\dom K(x-y) \omega(y) \, dy,
\end{align}
where 
\begin{align} \label{eq:kernel}
K(x) := \frac{1}{2\pi} \frac{x^\perp}{|x|^2}, \quad x^\perp = (-x_2,x_1).
\end{align}
If the vorticity is compactly supported, $\supp(\omega) \subset B_R(0)$, then we can write \cite[eq.(3.15)]{MajdaBertozzi}
\begin{align} \label{eq:asymptotic}
u(x) = K(x) \int_\dom \omega(y) \, dy + O(|x|^{-2}), \quad \forall \;|x|\ge 2R,
\end{align}
where the implied constant depends only on the $L^1$-norm of $\omega$ and on $R$. Equations \eqref{eq:kernelrep} and \eqref{eq:asymptotic} can be used to show that if $\int_\dom \omega(y) \, dy = 0$, then there exists a constant $C = C(\Vert \omega \Vert_{L^1},\Vert \omega \Vert_{L^\infty},R)$, such that 
\[
| u(x) | \le C(1+|x|^2)^{-1}.
\]
In particular, it follows that if $\omega \in L^1_x\cap L^\infty_x$ has compact support and $\int_\dom \omega \, dy = 0$, then $u \in L^2_x$. Using Fourier transforms, one can also show that there exists at most one solution $u\in L^2_x$, of the div-curl-system $\div(u) = 0$, $\curl(u) = \omega$, so that the representation \eqref{eq:kernelrep} is unique in this case.
\end{remark}

Having explained our motivation for the study of these structure functions, and their relevance for energy concentration and the convergence of approximate solution sequences, we now come to the following fundamental result for the present section. We state a novel identity, expressing the structure function $S_2(u;r)$ (equation \eqref{eq:structfun}) in terms of the vorticity.
\begin{lemma} \label{lem:structest}
Let $\omega \in L^1_x \cap L^\infty_x$, with $\supp(\omega)\subset \dom$ compact. Let $u = K\ast \omega$, so that $\div(u) = 0$, $\curl(u)= \omega$. Then 
\begin{gather} \label{eq:structvort} 
\begin{aligned}
\fint_{B_r(0)} &\int_{\dom} |u(x+h)-u(x)|^2\, dx \, dh
 &=
\int_{\dom}\int_{|h|\le r} 
\Sigma\left(\frac{|h|}{r}\right)  \omega(x)\omega(x+h) \, dh\, dx ,
\end{aligned}
\end{gather}
where, for $0<\rho\le 1$, we define $\Sigma(\rho):= (4\pi)^{-1}\left(|\log(\rho^2)| - 1 + \rho^2\right) \ge 0$. 
In particular, it follows that
\begin{gather*} 
\fint_{B_r(0)} \int_\dom |u(x+h)-u(x)|^2\, dx \, dh
\le 
\int_{\dom} \int_{B_r(0)} 
\left|\log\left(\frac{|h|}{r}\right)\right| |\omega(x)| |\omega(x+h)| 
 \, dh \, dx.
\end{gather*}
\end{lemma}

\begin{proof}
The detailed derivation of the expression \eqref{eq:structvort} relies on the explicit integral kernel representation $u = K\ast \omega$ (cp. Remark \ref{rem:kernel} above), and is given in appendix \ref{app:structfun}. It follows by combining the identities proved in Lemma \ref{lem:structintermediate} and Lemma \ref{lem:psipsiav}. The upper bound is immediate, because $\Sigma(\rho) \le |\log(\rho)|$ for all $\rho \le 1$.
\end{proof}

\begin{remark} \label{rem:L2remark}
Note that we do not need to require $\int_\dom \omega(y) \, dy$ to vanish in the statement of Lemma \ref{lem:structest}. This is not necessary, since the difference $u(x+h)-u(x)$ has vorticity $\omega(x+h)-\omega(x)$, so that $\int_\dom [\omega(x+h)-\omega(x)]\, dy = 0$, for any compactly supported vorticity. In particular, for any $h\in \dom$, we have $u(x+h)-u(x) \in L^2_x$, so that the left-hand side of \eqref{eq:structvort} is finite.
\end{remark}

\begin{remark}
By a similar argument as is used in the proof of Lemma \ref{lem:structest}, one can show that for vorticity $\omega$ with divergence-free velocity $u = K\ast\omega$, we have 
\[
\int_\dom |u(x) - [u]_r(x)| \, dx
\le 
C \Vert \omega \Vert_{\mathcal{M}} r,
\]
where $[u]_r(x) := \fint_{B_r(0)} u(x+h) \, dh$ denotes the local average of $u$ over a ball of radius $r$. This a priori estimate is closely related to the well-known strong $L^1$-compactness of approximate solution sequences of the incompressible Euler equations \cite{DipernaMajda1987a}.
\end{remark}

\subsection{Vorticity maximal function}

Since in this work, the vorticity $\omega \in \mathcal{M}$ is often a finite measure rather than an $L^1$-function, we first need to clarify the meaning of the vorticity maximal function \eqref{eq:vortmax} for finite measures. 

\begin{definition} \label{def:vortmax}
The \define{vorticity maximal function} 
\[
M_{r}(\omega):\,[0,\infty) \to [0,\infty), \quad r\mapsto M_r(\omega)
\]
of a measure $\omega \in \mathcal{M}$ is defined by
\begin{align} \label{eq:vortmaxdef}
M_r(\omega) := \sup_{x\in \dom} \int_{\overline{B}_r(x)} d|\omega|,
\end{align}
where $\overline{B}_r(x) := \{y \in \dom \, |\, |y-x|\le r\}$ denotes the closed ball of radius $r$ around $x$, and $|\omega| \ge 0$ is the variation of the finite signed measure $\omega$.
\end{definition}
Our Definition \ref{def:vortmax} agrees with the original definition \eqref{eq:vortmax} by Diperna and Majda \cite{DipernaMajda1987a,DipernaMajda1988}, if $\omega \in L^1$.
Our next aim in this section is to utilize the identity stated in Lemma \ref{lem:structest} to obtain an explicit estimate on the structure function in terms of the vorticity maximal function. This is achieved in Theorem \ref{thm:structmaxest} below. For its proof, we will need the following lemma.

\begin{lemma} \label{lem:vortmoll}
Let $\omega \in \mathcal{M}$ be a finite measure. For $\eta > 0$, let $\omega_\eta \in L^1$ denote mollification with a smooth mollifier $\rho_\eta$:
\[
\omega_\eta(x) := \int_\dom \rho_\eta(x-y) \, d\omega(y).
\]
Then $\Vert \omega_\eta \Vert_{L^1}\le \Vert \omega \Vert_{\mathcal{M}}$ and the vorticity maximal function $M_r(\omega_\eta)$ is bounded uniformly in $\eta > 0$:
\[
M_r(\omega_\eta) \le M_r(\omega), \quad \forall r > 0.
\]
\end{lemma}

\begin{proof}
The estimate $\Vert \omega_\eta \Vert_{L^1}\le \Vert \rho_\eta \Vert_{L^1}\Vert\omega\Vert_{\mathcal{M}} = \Vert\omega\Vert_{\mathcal{M}}$ is well-known. To estimate the vorticity maximal function, fix $r>0$ and $x\in \dom$. Let $\psi(z)$ be any smooth cut-off function, such that $\psi(z)\le 1_{B_r(x)}(z)$ for all $z\in \dom$. Then 
\begin{align*}
\int_{\dom} \psi(z)|\omega_\eta(z)| \, dz
&\le
\int_{\dom} \psi(z) \int_\dom \rho_\eta(z-y) d|\omega(y)| \, dz
\\
&=
\int_\dom \int_\dom \psi(z) \rho_\eta(z-y) d|\omega(y)| \, dz.
\end{align*}
Make the change of variables $h = z-y$, to find
\begin{align*}
\int_{\dom} \psi(z)|\omega_\eta(z)| \, dz
&\le
\int_\dom\rho_\eta(h) \left( \int_\dom \psi(y+h) d|\omega(y)|\right) \, dh
\\
&\le 
\int_\dom\rho_\eta(h) \left( \int_{B_r(x-h)} d|\omega(y)|\right) \, dh,
\end{align*}
where the last inequality follows from $\psi(y+h) \le 1_{B_r(x)}(y+h) = 1_{B_r(x-h)}(y)$. Thus, we have 
\[
\int_{\dom} \psi(z)|\omega_\eta(z)| \, dz
\le \int_\dom\rho_\eta(h) M_r(\omega) \, dh = M_r(\omega),
\]
for any smooth cut-off function $\psi(z)\le 1_{B_r(x)}(z)$. Choosing a sequence of such $\psi_k$, such that $\psi_k(z) \nearrow 1_{B_r(x)}(z)$ pointwise for all $z\in \dom$, the dominated convergence theorem now implies that 
\[
\int_{B_r(x)} |\omega_\eta(z)|\, dz
=
\lim_{k\to \infty} \int_\dom \psi_k(z) |\omega_\eta(z)|\, dz
\le M_r(\omega).
\]
Taking the supremum over $x\in \dom$ on the left-hand side, the claimed inequality follows.
\end{proof}

As a consequence of Lemma \ref{lem:structest} and Lemma \ref{lem:vortmoll}, we can now prove:
\begin{theorem} \label{thm:structmaxest}
Let the vorticity $\omega \in \mathcal{M}$. Assume that $\int_0^1 s^{-1}M_s(\omega) \, ds < \infty$, and let $u = K\ast \omega$ such that $\div(u) = 0$, $\curl(u) = \omega$. Then
\[
\fint_{B_r(0)} \int_\dom |u(x+h)-u(x)|^2\, dx \, dh
\le 
\Vert \omega \Vert_{\mathcal{M}} \int_0^r \frac{M_s(\omega)}{s} ds,
\]
for all $r\ge 0$.
\end{theorem}
\begin{proof}
The proof of this lemma will involve three steps:
\begin{itemize}
\item \textbf{Step 1:} Reduction to smooth $u,\omega \in C^\infty$, $\omega \in L^1_x\cap L^\infty_x$. 
\end{itemize}

Suppose the stated inequality holds for smooth $u$ and $\omega$. We want to show that it holds in general. Given $\eta > 0$, denote by $u_\eta$, $\omega_\eta$ the mollification of $u$ and $\omega$. Then, we clearly have $u_\eta, \omega_\eta \in C^\infty$, $\omega_\eta \in L^1_x \cap L^\infty_x$. If the inequality holds for this restricted class of functions, then
\[
\fint_{B_r(0)} \int_\dom |u_\eta(x+h)-u_\eta(x)|^2\, dx \, dh
\le 
\Vert \omega_\eta \Vert_{L^1} \int_0^r \frac{M_s(\omega_\eta)}{s} ds.
\]
By Lemma \ref{lem:structest}, we have $\Vert \omega_\eta \Vert_{L^1}\le \Vert \omega \Vert_{\mathcal{M}}$ and $M_s(\omega_\eta) \le M_s(\omega)$. Thus, the right hand side is bounded by 
\[
\Vert \omega_\eta \Vert_{L^1} \int_0^r \frac{M_s(\omega_\eta)}{s} ds
\le
\Vert \omega \Vert_{\mathcal{M}} \int_0^r \frac{M_s(\omega)}{s} ds.
\]
On the other hand, we can find a sequence $\eta_k \to 0$, such that the integrand on the left-hand side converges, $|u_{\eta_k}(x+h) - u_{\eta_k}(x)| \to |u(x+h)-u(x)|$ as $k \to \infty$, pointwise almost everywhere for $(x,h) \in \dom \times B_r(0)$. It follows from Fatou's lemma that 
\begin{align*}
\fint_{B_r(0)} \int_\dom |u(x+h)-u(x)|^2\, dx \, dh
&=
\fint_{B_r(0)} \int_\dom \liminf_{k\to \infty} |u_{\eta_k}(x+h)-u_{\eta_k}(x)|^2\, dx \, dh
\\
&\le \liminf_{k\to \infty} \fint_{B_r(0)} \int_\dom |u_{\eta_k}(x+h)-u_{\eta_k}(x)|^2\, dx \, dh
\\
&\le 
\Vert \omega \Vert_{\mathcal{M}} \int_0^r \frac{M_s(\omega)}{s} ds.
\end{align*}
Thus, the general estimate follows from the corresponding estimate for smooth $u$, $\omega$, and $\omega \in L^1_x\cap L^\infty_x$.
\begin{itemize}
\item \textbf{Step 2:} Reduction to compactly supported $\omega \in C^\infty_c$, $\omega \in L^1_x\cap L^\infty_x$. 
\end{itemize}
By Step 1, we may wlog assume that $u,\omega \in C^\infty$ and $\omega \in L^1_x\cap L^\infty_x$. We want to show that it is sufficient to prove the claimed inequality for compactly supported $\omega$. To see why, choose a smooth cut-off function $\rho(x)$ such that $\supp(\rho) \subset B_1(0)$, $\rho(x) \equiv 1$ for $|x|\le 1/2$, and set $\rho_R(x) := \rho(x/R)$. Denote $\omega_R(x) := \rho_R(x) \omega(x)$, and $u_R := K\ast \omega_R$. Note that 
\begin{align*}
|u(x) - u_R(x)| 
&\le
\int_\dom \frac{|\omega(y)|}{|x-y|} (1-\rho_R(y))\, dy
\\
&\le 
\int_{\dom} \frac{|\omega(y)|}{|x-y|}1_{[|y|\ge R/2]} \, dy.
\end{align*}
The last integrand is bounded by ${|\omega(y)|}/{|x-y|} \in L^1_y$ and converges pointwise to $0$ as $R\to \infty$. It thus follows from the Lebesgue dominated convergence theorem that the last integral converges to zero as $R\to \infty$, and hence $|u(x) - u_R(x)|\to 0$, for all $x\in \dom$. Assuming the claimed inequality holds for the compactly supported $\omega_R$, we once again find from Fatou's lemma
\begin{align*}
\fint_{B_r(0)} \int_\dom |u(x+h)-u(x)|^2\, dx \, dh
&=
\fint_{B_r(0)} \int_\dom \liminf_{R\to \infty} |u_{R}(x+h)-u_{R}(x)|^2\, dx \, dh
\\
&\le \liminf_{R\to \infty} \fint_{B_r(0)} \int_\dom |u_R(x+h)-u_R(x)|^2\, dx \, dh
\\
&\le 
\Vert \omega_R \Vert_{L^1} \int_0^r \frac{M_s(\omega_R)}{s} ds 
\\
&\le
\Vert \omega \Vert_{L^1} \int_0^r \frac{M_s(\omega)}{s} ds,
\end{align*}
where the last inequality follows trivially, since $|\omega_R| = \rho_R |\omega|\le |\omega|$.

\begin{itemize}
\item \textbf{Step 3:} Proof of the estimate for $\omega \in C^\infty_c$, $\omega \in L^1_x\cap L^\infty_x$.
\end{itemize}

By Steps 1 and 2, we may assume wlog that $\omega \in C^\infty_c$, $\omega \in L^1_x\cap L^\infty_x$. Denote $m(s) := M_s(\omega)$. By Lemma \ref{lem:structest}, we have an upper bound on 
\[
\fint_{B_r(0)} \int_\dom |u(x+h)-u(x)|^2\, dx \, dh,
\]
in terms of
\[
\int  |\omega(x)| 
\left(
\int_{B_r(0)} \left|\log\left(\frac{|z|}{r}\right)\right||\omega(x+z)| \, dz
\right)
\, dx
\le \Vert \omega \Vert_{L^1} 
\int_0^r \left|\log\left(\frac sr\right)\right| \, dm(s).
\]
In the last integral, $dm(s)$ denotes the Lebesgue-Stieltjes integral (such that $dm(s) = m'(s) \, ds$ for continuously differentiable $m(s)$). Since $|\log(s/r)| = -\log(s/r)$ for $s\in (0,r]$, we can integrate by parts to find
\[
\int_0^r \left|\log\left(\frac sr\right)\right| \, dm(s)
= 
\left[-\log\left(\frac sr\right) m(s)\right]_0^r
+ \int_0^r \frac{m(s)}{s} \, ds
=
 \int_0^r \frac{m(s)}{s} \, ds,
\]
having observed that the boundary terms conveniently cancel.\footnote{For the cancellation at $s=0$, we note the trivial bound $m(s) \le \Vert \omega\Vert_{L^\infty} \pi s^2$.} Thus,
\begin{align*}
\fint_{B_r(0)} \int |u(x+h)-u(x)|^2\, dx \, dh
\le
\Vert \omega \Vert_{L^1}  \int_0^r \frac{m(s)}{s} \, ds
=
\Vert \omega \Vert_{L^1}  \int_0^r \frac{M_s(\omega)}{s} \, ds.
\end{align*}
This is the claimed inequality.
\end{proof}

It follows from Theorem \ref{thm:structmaxest}, that any vorticity decay that is slightly better than logarithmic implies a uniform decay of the structure function. In particular, we have the following corollary, which together with Proposition \ref{prop:compactness} provides an alternative proof of \cite[Theorem 3.1]{DipernaMajda1987a} (under slightly relaxed assumptions).\footnote{In \cite[Theorem 3.1]{DipernaMajda1987a} an additional assumption $\int |\log|x|| |\omega(x)|\, dx < \infty$ was required to obtain an analogous non-concentration result.}

\begin{corollary} \label{cor:logdecay}
If the vorticity $\omega \in H^{-1}_{\loc}\cap \mathcal{M}$, $u = K\ast \omega$, and if 
\[
M_r(\omega)
\le {C} |\log(r)|^{-\beta},
\]
for some $\beta>1$, then 
\[
\fint_{B_r(0)} \int_\dom |u(x+h)-u(x)|^2\, dx \, dh
\le 
\frac{C}{\beta-1} \Vert \omega \Vert_{\mathcal{M}} |\log(r)|^{1-\beta},
\]
gives uniform control on the structure function.
\end{corollary}

Under the assumption of an algebraic vorticity decay, we obtain
\begin{corollary} \label{cor:algebraicdecay}
If the vorticity $\omega \in H^{-1}_{\loc}\cap \mathcal{M}$, $u = K\ast \omega$, and if
\[
M_r(\omega) 
\le C r^\beta,
\]
for some $\beta>0$, then 
\[
\fint_{B_r(0)} \int_\dom |u(x+h)-u(x)|^2\, dx \, dh
\le 
\frac{C}{\beta} \Vert \omega \Vert_{\mathcal{M}} \, r^\beta.
\]
\end{corollary}

\section{Vortex-blob method} \label{sec:vortexblob}

The vorticity formulation of the incompressible Euler equations \eqref{eq:vorticity} express the advection of vorticity by the incompressible velocity field $u(x,t) = (K\ast \omega)(x,t)$, where $K(x)$ is the kernel given by \eqref{eq:kernel}. The vortex method, as introduced by Chorin and Bernard \cite{Chorin1973}, approximates the vorticity field by a sum of weighted Dirac delta measures,
\[
\omega(x,t) = \sum_{j=1}^N \xi_j \dirac(x - X_j(t)),
\]
at vortex positions $X_1(t), \dots, X_N(t) \in \dom$ and with vorticities $\xi_1,\dots, \xi_N$. The vortices are advected by a regularized velocity field, which is determined as follows: We fix a mollifier (vortex blob function) $\phi\in C^2$ \cite{LiuXin1995} such that
\[
\phi(x) = f(|x|) \ge 0, 
\quad 
\int_{|x|<1} \phi(x) \, dx,
\quad
\int_{|x|\le 1/2} \phi \, dx \ge \frac 12.
\]
Given a blob size $\epsilon > 0$, we denote $\phi_\epsilon(x) := \epsilon^{-2} \phi(x/\epsilon)$. The regularized velocity is given by 
\begin{align} \label{eq:Kdelta}
u^\epsilon(x,t) := (K_\epsilon \ast \omega)(x,t), \quad K_\epsilon(x) := (\phi_\epsilon \ast K)(x).
\end{align}
Given initial vorticity data for the incompressible Euler equations $\omega_0 \in H^{-1} \cap \mathcal{M}$ and an initial approximation
\[
\omega_0(x) \approx \sum_{j=1}^N \xi_j \dirac(x-X_j^{(0)}),
\]
the vortex-blob method solves the following system of differential equations \cite{LiuXin1995}
\begin{gather}
\frac{d}{dt} X_j(t) = \sum_{k=1}^N \xi_k K_\epsilon(X_j(t) - X_k(t)), \quad X_j(0) = X^{(0)}_j, \label{eq:vortexblob}
\end{gather}
The approximate solution sequence $u^\epsilon$ is obtained from the evolution of the vortices by convolution with the smoothened kernel $K_\epsilon$ \eqref{eq:Kdelta}:
\begin{gather}
u^\epsilon(x,t) = \sum_{j=1}^N \xi_j K_\epsilon(x - X_j(t)).
\end{gather}
Note that if we define the discrete vorticity \cite{LiuXin1995}
\[
\omega_\epsilon(x,t) := \sum_{j=1}^N \xi_j \dirac(x-X_j(t)),
\]
then 
\[
\partial_t \omega_\epsilon + \div(\omega_\epsilon u^\epsilon) = 0,
\]
holds, in the sense of distributions.

As in \cite{LiuXin1995}, in addition to $u^\epsilon$ (a smooth vector field) and $\omega_\epsilon$ (a sum of Dirac deltas), we will denote
\[
\omega^\epsilon(x,t) := \sum_{j=1}^N \xi_j \phi_\epsilon(x-X_j(t)), 
\quad
u_\epsilon(x,t) := \sum_{j=1}^N \xi_j K(x-X_j(t)).
\]
Then, $\curl(u^\epsilon(x,t)) = \omega^\epsilon(x,t)$ and $\curl(u_\epsilon(x,t)) = \omega_\epsilon(x,t)$ in the sense of distributions.

\begin{remark}
A popular choice for the vortex blob function $\phi$, which is used to regularize the kernel $K$, is given by
\[
\phi(x) = \frac{1}{\pi} \left(\frac{1}{1+|x|^2}\right)^2,
\]
leading to the regularized Kernel 
\[
K_\epsilon(x) = \frac{1}{2\pi} \frac{x^\perp}{|x|^2 + \epsilon^2},
\]
as used by Krasny \cite{Krasny1987}.
\end{remark}

Clearly, the vortex-blob method, as described above, depends on the number of discrete vortices $N$, and on the initialization of the vorticities and vortex positions $(\xi_j,X_j^{(0)})$ in addition to the regularization parameter $\epsilon > 0$. In general, there are many possible choices for this initialization. We propose the following 

\begin{definition} \label{def:goodapprox}
Let $\omega_0 \in \mathcal{M} \cap H^{-1}$ be initial vorticity data for the incompressible Euler equations. Given a positive sequence $\epsilon \to 0$, we say that a family 
\[
\{t\mapsto \omega_\epsilon(t)\, | \, t\in [0,T]\},
\]
of weak-$\ast$ measurable mappings $\omega_\epsilon: \, [0,T] \to \mathcal{M}$ is a \define{good vortex-blob approximation} of the vorticity equation with initial data $\omega_0$, if each $\omega_\epsilon(t)$ is of the form 
\[
\omega_\epsilon(x,t) = \sum_{j=1}^N \xi_j \dirac(x - X_j(t)),
\]
with $N = N(\epsilon) \to \infty$ as $\epsilon \to 0$, the $\xi_1, \dots, \xi_N \in \mathbb{R}$ are constant in time, the vortex positions $X_j(t)$ satisfy the vortex-blob equations \eqref{eq:vortexblob} with parameter $\epsilon>0$, and the initial data $\omega_0$ is approximated by $\omega_\epsilon(t=0)$ in the following sense:
\begin{itemize}
\item $\omega_\epsilon(t=0) \weakstarto \omega_0$ in $\mathcal{M}$ as $\epsilon \to 0$,
\item $\Vert \omega_\epsilon(t=0) \Vert_{\mathcal{M}} = \sum_{j=1}^{N(\epsilon)} |\xi_j| \le \Vert \omega_0 \Vert_{\mathcal{M}}$ for all $\epsilon > 0$,
\item $H^\epsilon(\omega_\epsilon(t=0)) \to H(\omega_0)$ as $\epsilon \to 0$,
\end{itemize}
where 
\[
H^\epsilon(\omega_\epsilon) 
:= 
\frac{-1}{4\pi}\int_{D\times D} \log(|x-y|^2 + \epsilon^2) \omega_\epsilon(x)\omega_\epsilon(y) \, dx \, dy, 
\quad \text{for }
\epsilon > 0,
\]
and $H(\omega_0) = \Vert \omega_0 \Vert_{H^{-1}}$ (corresponding formally to the limit $\epsilon \to 0$).
\end{definition}

According to Definition \ref{def:vortmax}, the maximal vorticity function $M_r(\omega(t))$ of a discrete vorticity distribution $\omega(t) = \sum_{j=1}^N \xi_j \dirac(x-X_j(t))$ is given by
\begin{align}
M_r(\omega(t)) := \sup_{x\in D} \sum_{|X_j(t)-x|\le r} |\xi_j|.
\end{align}
This agrees with the definition of \cite{LiuXin1995}.

We can now formulate the following conditional convergence result for the vortex-blob method:

\begin{theorem} \label{thm:strongconv}
Suppose that the initial data $\omega_0 \in \mathcal{M} \cap H^{-1}$ is a finite, compactly supported Borel measure on $\dom$, $\int_\dom \, d\omega_0 = 0$. Let $u^\epsilon$ be the approximate solution sequence generated by a good vortex approximation with smoothing parameter $\epsilon \to 0$. If there exist constants $C$, $\beta > 1$, such that (uniformly in $\epsilon > 0$)
\[
\sup_{t\in [0,T]} M_r(\omega^\epsilon(t)) \le 
C|\log(r)|^{-\beta},
\quad \forall \, r \in (0,1/2), \; \epsilon > 0,
\]
and if there exists $R>0$ such that $\supp(\omega^\epsilon(t)) \subset B_R(0)$ for all $t\in [0,T]$, then there exists $u\in L^2_{t,x}$, and a subsequence $\epsilon \to 0$, such that $u^\epsilon \to u$ converges strongly in $L^2_{t,x}$. The structure functions exhibit uniform decay
\[
\sup_{t\in [0,T]} S_2(u^\epsilon(t);r) \le 
C |\log(r)|^{\beta-1},
\quad \forall \, r \in (0,1/2), \; \epsilon > 0,
\]
and the limit $u$ is an energy-conservative weak solution of the incompressible Euler equations. In particular, no energy concentration occurs in the limit $u^\epsilon \to u$.
\end{theorem}

The proof of Theorem \ref{thm:strongconv} depends on several technical, but mostly straight-forward estimates on the approximate solution sequence, building on the results of Liu and Xin \cite{LiuXin1995}. In order not to disrupt the flow of this work, we have included the detailed proof in appendix \ref{app:vortexblob}. 

\begin{remark}
Theorem \ref{thm:strongconv} shows that even very weak \emph{logarithmic control} on the vorticity maximal function implies the absence of energy concentration phenomena. We remark in passing that if the sequence $\omega^\epsilon$ satisfies a uniform algebraic bound
\[
\sup_{t\in [0,T]} M_r(\omega^\epsilon(t)) \lesssim r^{\alpha},
\]
for some $\alpha > 0$, then the structure functions can be bounded by an algebraic upper bound
\[
\sup_{t\in [0,T]} S_2(u^\epsilon(t);r) \lesssim r^{\alpha/2}. 
\]
This straightforward extension is based on Corollary \ref{cor:algebraicdecay}. In fact, such an algebraic upper bound on the structure functions has been observed in numerical experiments based on spectral methods \cite{LMP2019,LMP2020}.
\end{remark}

\begin{remark}
The formulation of Theorem \ref{thm:strongconv} in terms of a uniform decay of the vorticity maximal function will be particularly convenient as a practical tool to analyse our numerical experiments, because the vortex-blob method outputs the vorticity $\omega_\epsilon$, rather than the velocity $u^\epsilon$. 
\end{remark}

\begin{remark}
In the work of Liu and Xin \cite{LiuXin1995}, a \emph{weak} convergence result (rather than strong convergence as in Theorem \ref{thm:strongconv}) was obtained under a weaker uniform decay condition on the vorticity maximal function. Their convergence result does not rule out energy concentration in the limit. As will be seen in Section \ref{sec:numerical}, numerical experiments indicate that the assumptions of Theorem \ref{thm:strongconv} are fulfilled in practice, at least for the initial data considered in that section. Under this condition, Theorem \ref{thm:strongconv} shows that the corresponding approximate solution sequence is compact in $L^2_{t,x}$, and hence no energy concentration can take place in the limit $\epsilon \to 0$. 
\end{remark}

\section{Numerical Experiments} \label{sec:numerical}

The connection between a uniform decay of the vorticity maximal function $M_r(\omega^\epsilon)$ and a uniform decay of the structure functions $S_2(u^\epsilon;r)$ has been explained in Section \ref{sec:structfun} (cp. Theorem \ref{thm:structmaxest}). As a consequence, in Corollary \ref{cor:logdecay} and Corollary \ref{cor:algebraicdecay} we have shown that suitable decay of the vorticity maximal function implies uniform decay of the structure function. Finally, according to Proposition \ref{prop:compactness}, such a uniform algebraic decay gives compactness in $L^2_{\mathrm{loc}}$. Focusing on the case of algebraic decay, we can thus summarize our results schematically as follows:
\begin{align} \label{eq:implications}
M_r(\omega^\epsilon(t)) \le Cr^\alpha
\; \Rightarrow \;
S_2(u^\epsilon(t);r) \le Cr^{\alpha/2}
\; \Rightarrow \;
\text{ $u^\epsilon$ precompact in $L^2_tL^2_{x,\mathrm{loc}}$.}
\end{align}
The first two bounds are understood to hold uniformly for all $t\in [0,T]$. As explained in remark \ref{rem:concentration}, a lack of strong $L^2_{\mathrm{loc}}$-compactness in approximate solution sequences $u^\epsilon$ with vortex sheet initial data can only arise due to concentration effects in the limit $\epsilon \to 0$. In particular, considering the implications of \eqref{eq:implications} in contraposition, any concentration in the limit $\epsilon \to 0$ would rule out a uniform bound of the form $M_r(\omega^\epsilon) \le Cr^\alpha$. 

The central question of this work is whether there is any numerical evidence that such concentration phenomena appear in the evolution of vortex-sheets. To shed new light on this question, the present section focuses on numerical experiments conducted with the vortex-blob method reviewed in the previous Section \ref{sec:vortexblob}, for which the analogue of the implications \eqref{eq:implications} have been stated in Theorem \ref{thm:strongconv}. In particular, if the regularized evolution of vortex sheets by the vortex-blob method (as studied in \cite{Krasny1987}) does indeed exhibit concentrations in the limit $\epsilon \to 0$, then our numerically obtained solutions cannot satisfy any uniform bound of the form $M_r(\omega^\epsilon) \le Cr^\alpha$. This motivates our detailed study of the behaviour of the numerically determined maximal vorticity functions $M_r(\omega^\epsilon(t))$ presented in the remainder of this work.

\subsection{Setup} \label{sec:numerical:setup}

We consider two types of initial data, which we will refer to as the ``loaded wing'' and the ``fuselage flap'' configurations, following the terminology of \cite{Krasny1987}. In both cases, the initial vorticity data $\omega_0$ is given by a flat vortex sheet along the straight curve from $(-1,0)$ to $(1,0)$ in the plane $\R^2$. As in \cite{Krasny1987}, we parametrize the initial vortex sheet by a curve $z_0: [0,\pi] \to \R^2$, $z_0(\alpha) := (-\cos(\alpha),0)$. The regularized vortex sheet dynamics with blob-size $\epsilon>0$ then leads to the following evolution equation for this parametrization in time
\begin{align} \label{eq:birkhoffrott}
\partial_t z(\alpha,t) := \int_0^\pi K_\epsilon(z(\alpha,t)-z(\tilde{\alpha},t)) \Gamma'(\tilde{\alpha}) \, d\tilde{\alpha}.
\end{align}
Here $\Gamma' = d\Gamma/d\alpha$ and $\alpha \mapsto \Gamma(\alpha)$ measures the circulation along the vortex sheet. The circulation is specified at the initial time, so that the integration of any smooth test function $\phi \in C^\infty(\R^2)$ against the measure $\omega_0 \in \mathcal{M}$ can be written in the form
\[
\int_\D \phi(x) \, d \omega_0(x) = \int_0^\pi \phi(z_0(\alpha)) \Gamma'(\alpha) \, d\alpha.
\]
For the loaded wing configuration, we have $\Gamma = \Gamma_{\mathrm{lw}}$, where
\begin{align} \label{eq:lw}
\Gamma_{\mathrm{lw}}(\alpha) = \sin(\alpha).
\end{align}
For the fuselage flap configuration, we define $\Gamma = \Gamma_{\mathrm{ff}}$, where 
\begin{align} \label{eq:ff}
\Gamma_{\mathrm{ff}}(\alpha) 
=
\begin{cases}
\sin(\alpha), & |\cos(\alpha)|>0.7, \\
P(\cos(\alpha)), &  |\cos(\alpha)| \le 0.7, 
\end{cases}
\end{align}
Here, $P: [-0.7,0.7] \to \mathbb{R}$, $x\mapsto P(x)$ is a cubic spline on the intervals between $0,\pm 0.3, \pm 0.7$, chosen such that $\alpha \mapsto \Gamma(\alpha), \,\Gamma'(\alpha)$ are continuous (providing spline boundary conditions) and enforcing that $\Gamma(\cos^{-1}(\pm 0.3)) = 2$, $\Gamma(\cos^{-1}(0)) = 1.4$, where $\cos^{-1}$ is the inverse of $\cos: [0,\pi] \to [-1,1]$. 
In practice, $P(x)$ has been determined as follows: The cubic spline interpolation conditions are
\[
P(\pm 0.7) = \sqrt{1-(0.7)^2}, \quad P(\pm 0.3) = 2, \quad P(0) = 1.4,
\]
In addition, we enforce the spline boundary conditions
\[
P'(x_0) = -x_\pm/\sqrt{1-x_\pm^2}, \quad \text{at } x_\pm=\pm 0.7.
\]
This leads to a piece-wise polynomial function $P(x)$ of the form
\[
P(x) = a_0 + a_1 |x| + a_2 |x|^2 + a_3 |x|^3,
\]
on each subinterval. For $|x|\in [0,0.3]$ and $|x|\in [0.3,0.7]$, we find, respectively,
\begin{gather*}
\left\{
\begin{aligned}
a_0 &= 1.4, \\ 
a_1 &= 0.0,\\
a_2 &= 20.0,\\
a_3 &= -44.444444444444443,
\end{aligned}
\right.
\quad
\left\{
\begin{aligned}
a_0 &= -0.868873730864876,\\
a_1 &= 22.190937843818809,\\
a_2 &= -52.310461263296190,\\
a_3 &= 34.056810793181128.
\end{aligned}
\right.
\end{gather*}
The resulting circulation $\Gamma$, as a function of the horizontal variable $x = -\cos(\alpha)$ is shown in Figure \ref{fig:initialdata} (A), for both the loaded wing (dashed line) and the fuselage flap (solid line) configurations. Also shown in Figure \ref{fig:initialdata} (B) is the vortex sheet strength $\sigma = -d\Gamma/ds$ where $s$ is the arclength. Note that the vortex sheet strength measures the jump of the velocity corresponding to the vortex sheet \cite{Krasny1987}.

\begin{figure}[H]
\centering
\begin{subfigure}{.45\textwidth}
\includegraphics[width=\textwidth]{{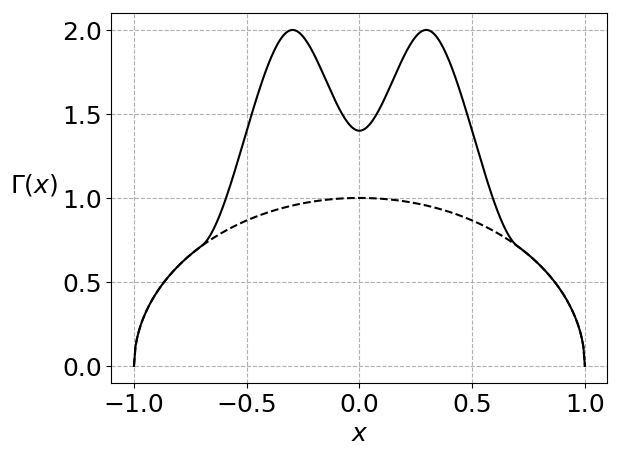}}
\caption{Circulation $\Gamma$}
\end{subfigure}
\hspace{5pt}
\begin{subfigure}{.45\textwidth}
\includegraphics[width=\textwidth]{{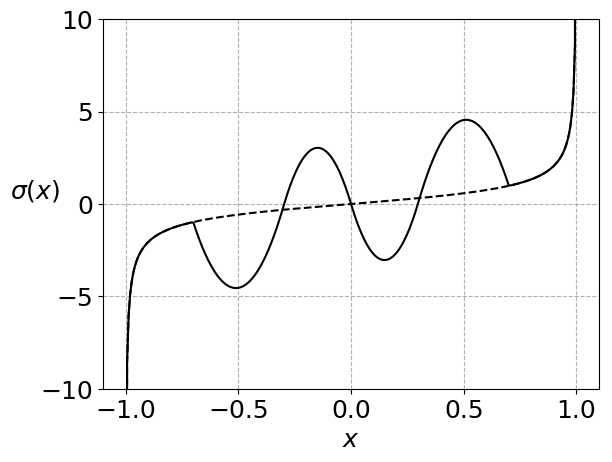}}
\caption{Vortex sheet strength $\sigma = -{d\Gamma}/{ds}$}
\end{subfigure}
\caption{Illustration of the initial circulation (left) and vortex sheet strength (right) for the loaded wing (dashed line) and fuselage flap (solid line) configurations considered in this section (cp. with \cite[Figure 1]{Krasny1987}).}
\label{fig:initialdata}
\end{figure}

Let $\alpha_j := \pi j/N$ for $j=0,\dots,N$. Employing the vortex-blob method, we approximate the right-hand side of the parametrized vortex sheet evolution equation \eqref{eq:birkhoffrott} by the trapezoidal rule
\[
\int_0^\pi K_\epsilon(z(\alpha,t)-z(\tilde{\alpha},t)) \Gamma'(\tilde{\alpha}) \, d\tilde{\alpha}
\approx
\sum_{j=0}^N K_\epsilon(z(\alpha,t)-z(\alpha_j,t)) w_j,
\]
where $w_j = \Gamma'(\alpha_j)\pi/N$ for $j=1,\dots, N-1$ and $w_j =  \Gamma'(\alpha_j)\pi/(2N)$ for $j=0,N$. With $z_j(t) \approx z(\alpha_j,t)$, this leads to the vortex-blob evolution equation
\[
\frac{d}{dt} z_k(t) = \sum_{j=0}^N K_\epsilon(z_k(t)-z_j(t)) w_j.
\]
with initial data $z_j(t=0) = -\cos(\alpha_j)$ \cite{Krasny1987}.

In our numerical implementation, the system of ordinary differential equations for the $z_j(t)$ are solved using RK4 time-stepping, with a fixed prescribed time-step $\Delta t$. The right-hand side of the evolution equations are computed using a simple, direct evaluation of the sum (\texttt{openMP} parallelized). In particular, we have not made use of the ``point-insertion technique'' employed in \cite{Krasny1987}, nor of a potentially more efficient $O(N\, \log N)$ tree-based evaluation of the right-hand side terms \cite{Draghicescu1995}. The accuracy of the time-stepping has been measured by tracing the constants of motion of the vortex-blob method
\[
H^\epsilon = -\frac1{4\pi} \sum_{i,j=0}^N w_iw_j\log\left(|z_i-z_j|^2 + \epsilon^2\right), \quad W = \sum_{j=0}^N w_j z_j,
\]
over time; the time-step is chosen small enough to ensure conservation of these quantities, with a relative error $\le 5\cdot 10^{-4}$ over the simulation.

For both cases, we will consider the decay of the vorticity maximal function over time, i.e. we will track the time-dependent functions
\[
r \mapsto M_r(\omega_\epsilon(t)) := \sup_{z\in \dom} \sum_{|z_j(t) - z| \le r} |w_j|.
\]
As explained in Section \ref{sec:vortexblob}, a uniform algebraic decay of $M_r(\omega_\epsilon(t))$ (uniform in $\epsilon \to 0$ and $t \in [0,T]$) implies a lack of concentration in the limit $\epsilon \to 0$, and strong convergence of the approximate velocity fields $u^\epsilon \to u$ to an energy-preserving weak solution $u$ of the incompressible Euler equations. We propose two numerical tests to track the evolution of $M_r(\omega_\epsilon(t))$ numerically: a global approach and a local approach.

\subsubsection{Global tracking of vorticity maximal function:} \label{sec:global}

The global evaluation provides an approximation of $M_r(\omega_\epsilon(t))$. This approximation is obtained as follows: Given a time $t \in [0,T]$ and discrete vorticity data $w_j$, $z_j(t)$ obtained by the vortex blob method with vorticity
\[
\omega_\epsilon(t)(z) = \sum_{j=0}^N w_j \dirac(z - z_j(t)),
\]
we first fix a square grid of a total length size $L$, centered at a point $(x_c,y_c) \in \R^2$, such that all vortices are contained in the square, i.e.
\[
z_j(t) \in [x_c-L,x_c+L] \times [y_c-L,y_c+L] \quad \forall \, j=0,\dots, N.
\]
In practice, we first set $x_c = (x_{\mathrm{min}} + x_{\mathrm{max}})/2$, $y_c = (y_{\mathrm{min}} + y_{\mathrm{max}})/2$, where $x_{\mathrm{min},\mathrm{max}}$ are the minimum/maximum of all x-coordinates of the vortex positions $z_j(t)$, and similarly for $y$. With this choice of the centre $(x_c,y_c)$, the length $L$ is then chosen minimal to fit the entire distribution in the square box surrounding $(x_c,y_c)$. Let $(x_0,y_0)=(x_c-L,y_c-L)$ denote the lower left corner of the square. Given a (large) subdivision number $N_d$, we then subdivide the square into a regular $N_d\times N_d$-grid 
\[
\mathcal{G} = \{g_{k,\ell} \, | \, k,\ell=0,\dots,N_d\}, \quad g_{k,\ell} := (x_k,y_\ell), \quad (x_k,y_\ell) = (x_0,y_0) + \frac{2L}{N_d}(k,\ell).
\]
This yields a partition into small squares $Q_{k,\ell} := [x_{k},x_{k+1}) \times [y_{\ell},y_{\ell+1})$ with side length $2L/N_d$. To obtain an approximation of the maximal vorticity at length scale $r_{\mathrm{min}} = 2L/N_d$, we bin the vortices in the square boxes $Q_{k,\ell}$ by defining
\[
m^{(0)}_{k,\ell} := \sum_{z_j(t) \in Q_{k,\ell}} |w_j|, \quad \text{and approximate} \quad M_{r_\mathrm{min}}(\omega_\epsilon(t)) \approx \max_{k,\ell} m^{(0)}_{k,\ell}.
\]
To obtain an approximation of the vorticity at larger scales $2^n r_{\mathrm{min}}$ for $n > 0$, we define
\[
m^{(n)}_{k,\ell} := \sum_{|k'-k|<2^n}\sum_{|\ell'-\ell|<2^n} m^{(0)}_{k+k',\ell+\ell'},
\]
and approximate $M_{2^n r_\mathrm{min}}(\omega_\epsilon(t)) \approx \max_{k,\ell} m^{(n)}_{k,\ell}$. To analyse the output of our simulations, we have chosen $N_d = 4096$, corresponding to $4096^2$ bins and $r_{\mathrm{min}} \approx 5\cdot 10^{-4}$.

\subsubsection{Local tracking of vorticity maximal function:} \label{sec:local}

As another quantity of interest, we propose to track the evolution of the vorticity maximal function at local points along the vortex sheet. To this end, we recall that the numerical scheme outlined in Section \ref{sec:numerical:setup} provides an approximation $z_j = z(\alpha_j,t)$ of the evolution of the parametrization $\alpha \mapsto z(\alpha,t)$ of the vortex sheet. For the local tracking, given $\alpha \in [0,\pi]$, we find the discretized point $\alpha_j \approx \alpha$ closest to $\alpha$, and track the evolution of $z_j(t) \approx z(\alpha,t)$ over time. At any given time $t \in [0,T]$, we can then compute 
\[
r \mapsto \sum_{|z_i(t)-z_j(t)| \le r} |w_j|. 
\]
This provides a local measure of the evolution of the vorticity concentration at individual points of interest along the vortex sheet, providing further detailed insight into the behaviour of the vorticity at a local level.

\begin{figure}[H]
\centering
\begin{subfigure}{.45\textwidth}
\includegraphics[width=\textwidth]{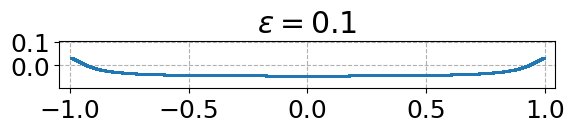}
\caption{$t = 0.1$}
\end{subfigure}
\begin{subfigure}{.45\textwidth}
\includegraphics[width=\textwidth]{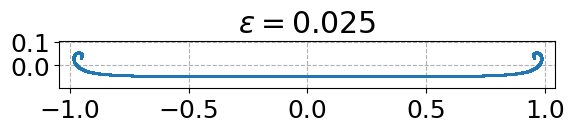}
\caption{$t = 0.1$}
\end{subfigure} 
\\[5pt]
\begin{subfigure}{.45\textwidth}
\includegraphics[width=\textwidth]{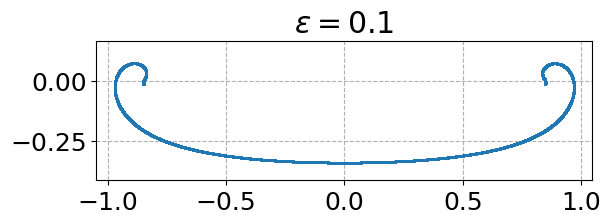}
\caption{$t = 0.8$}
\end{subfigure}
\begin{subfigure}{.45\textwidth}
\includegraphics[width=\textwidth]{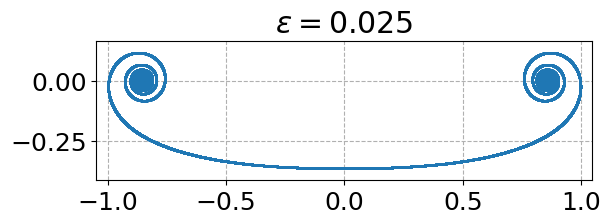}
\caption{$t = 0.8$}
\end{subfigure} 
\\[5pt]
\begin{subfigure}{.45\textwidth}
\includegraphics[width=\textwidth]{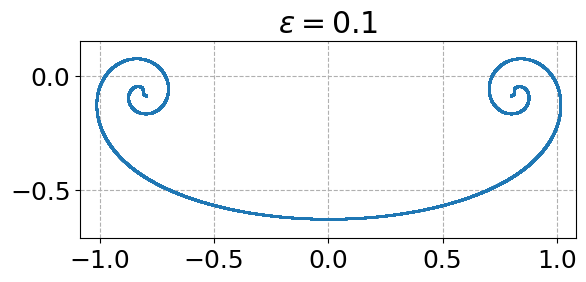}
\caption{$t = 1.6$}
\end{subfigure}
\begin{subfigure}{.45\textwidth}
\includegraphics[width=\textwidth]{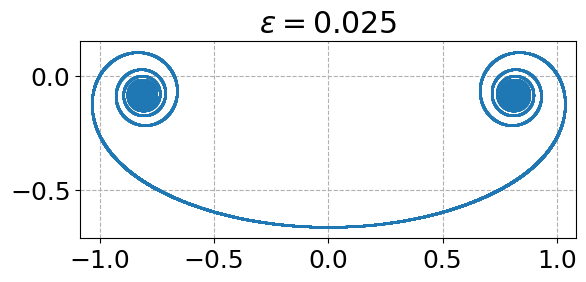}
\caption{$t = 1.6$}
\end{subfigure} 
\caption{Loaded-wing configuration: Evolution of vortex sheet with regularization parameter $\epsilon = 0.1$ (left) and $\epsilon = 0.025$ (right) for $t = 0.1$, $0.8$, $1.6$. Computed with the vortex-blob scheme using $N=100'001$ vortices.}
\label{fig:evolution_lw}
\end{figure}

\subsection{Loaded wing configuration} \label{sec:loadedwing}

As has been demonstrated in computations first performed by Krasny \cite{Krasny1987}, the loaded-wing configuration exhibits much more regular behaviour than the fuselage flap configuration. The loaded-wing configuration exhibits a simple (approximately self-similar) vortex sheet roll-up at the vortex tips. This behaviour appears to be qualitatively very similar to the dynamics that has been observed for a single signed periodic vortex sheet \cite{Krasny1986,Krasny1990}. Before considering the fuselage flap configuration in the next section, we here present the results of our simulations for the loaded-wing case, which will serve as a reference case for which concentration phenomena are not expected to occur. We note that this loaded wing initial data is of the mirror symmetric form for which the existence of a weak solution has been established in \cite{LNZ2001}.

Using the vortex-blob method with initialization described in detail in the last section, we have computed the evolution of the vortex sheet up to time $t=1.6$, at different values of the regularisation parameter $\epsilon = 0.1$, $0.05$ and $0.025$. For each regularisation, we initialize the vortex sheet with a total of $N=100'001$ discrete vortices. A constant time-step of $\Delta t = 0.005$ has been chosen for these simulations. Representative plots of the vortex sheet evolution with different regularization parameters $\epsilon = 0.1$, $0.025$ are shown in Figure \ref{fig:evolution_lw}.

Figure \ref{fig:evolution_lw} reproduces the evolution depicted in \cite[Figure 2]{Krasny1987}. We emphasize that these plots do not show an interpolated curve through the discrete vortices; rather, each vortex is represented as a single dot. The appearance of a continuous curve is due to the density of vortex points. Next, we numerically compute the evolution of the discrete vorticity maximal function $M_r(\omega_\epsilon(t))$ for different values of $r$ and at different times $t$, using the global numerical method explained in Section \ref{sec:global}. The results at times $t=0.1$, $0.4$, $0.8$ and $1.6$ are shown in Figure \ref{fig:vortmax_lw}.

\begin{figure}[H]
\centering
\begin{subfigure}{.45\textwidth}
\includegraphics[width=\textwidth]{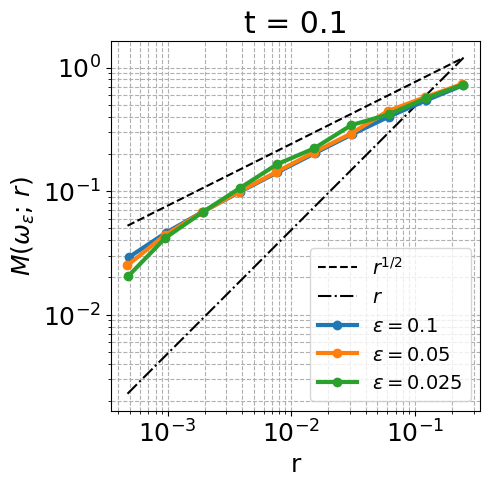}
\caption{$t = 0.1$}
\end{subfigure}
\begin{subfigure}{.45\textwidth}
\includegraphics[width=\textwidth]{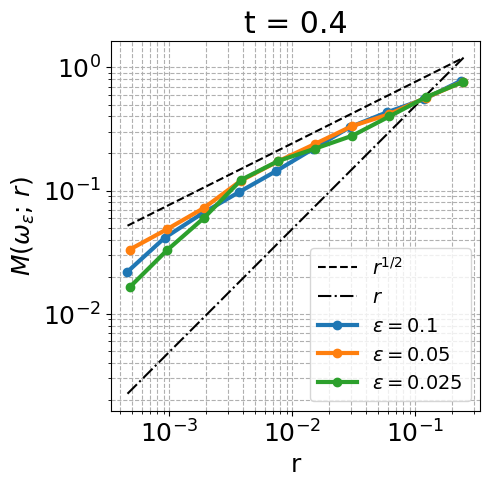}
\caption{$t = 0.4$}
\end{subfigure} 
\\[5pt]
\begin{subfigure}{.45\textwidth}
\includegraphics[width=\textwidth]{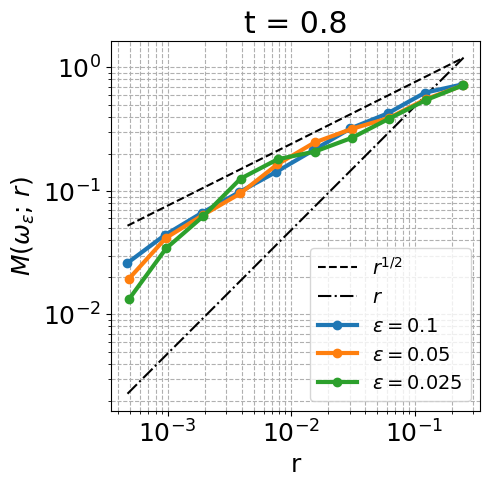}
\caption{$t = 0.8$}
\end{subfigure}
\begin{subfigure}{.45\textwidth}
\includegraphics[width=\textwidth]{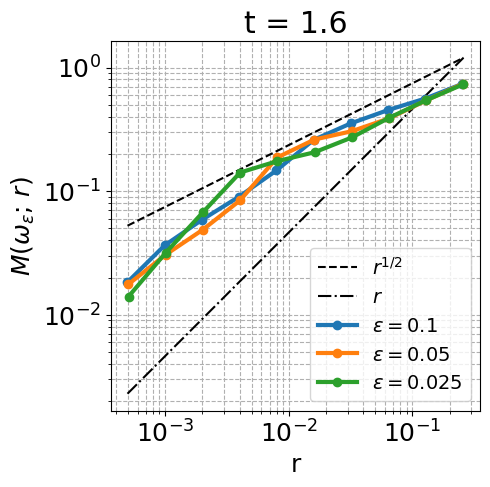}
\caption{$t = 1.6$}
\end{subfigure} 
\caption{Loaded-wing configuration: Evolution of vorticity maximal function $M_r(\omega_\epsilon(t))$ with regularization parameter $\epsilon \in \{ 0.025,0.05,0.1 \}$ at $t = 0.1$, $0.4$, $0.8$ and $1.6$. Computed with the vortex-blob scheme using $N=100'001$ vortices.}
\label{fig:vortmax_lw}
\end{figure}

Figure \ref{fig:vortmax_lw} shows an approximate scaling of the vorticity maximal function of $M_r \sim r^{1/2}$, which is to be expected due to the $1/\sqrt{1-x^2}$ singularities of the vortex sheet strength at the vortex sheet tips. More precisely, we can explicitly compute the vorticity maximal function for the initial data (e.g. centered at the right tip $z = (1,0)$):
\[
\int_{B_r(z)} d |\omega(y)|
=
\int_{1-r}^1 \frac{dx}{\sqrt{1-x^2}}
=
\int_0^r \frac{du}{\sqrt{2u}} + O(r)
=
\frac{1}{2\sqrt{2}}\sqrt{r} + O(r).
\]
Figure \ref{fig:vortmax_lw} shows that this initial scaling is approximately conserved over time, suggesting an algebraic bound for the vorticity maximal function also at later times for the loaded wing configuration. 

We can further confirm this observation by locally tracking the vorticity concentration at individual points. For the loaded wing configuration, the vorticity is clearly expected to be maximal near the wing tips. We show the evolution of the local vorticity maximal function (cp. Section \ref{sec:local}) at points $\alpha = 0.9\pi$, $0.965\pi$ and $\pi$ in Figures \ref{fig:localtracking_lw} and \ref{fig:vortmaxlocal_lw}.

\begin{figure}[H]
\centering
\begin{subfigure}{.49\textwidth}
\includegraphics[width=\textwidth]{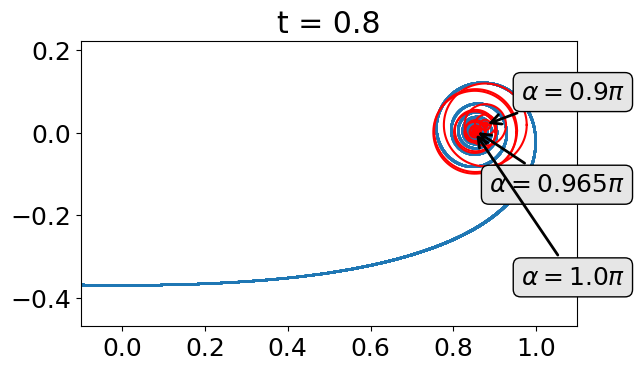}
\caption{$t = 0.8$}
\end{subfigure} 
\begin{subfigure}{.49\textwidth}
\includegraphics[width=.95\textwidth]{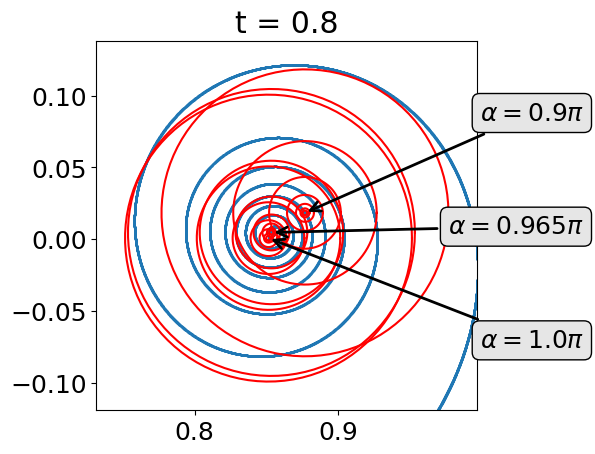}
\caption{close-up view}
\end{subfigure}
\\[5pt]
\begin{subfigure}{.49\textwidth}
\includegraphics[width=\textwidth]{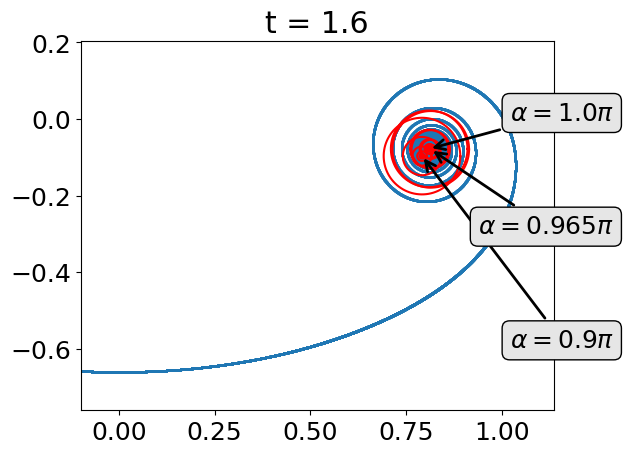}
\caption{$t=1.6$}
\end{subfigure} 
\begin{subfigure}{.49\textwidth}
\includegraphics[width=\textwidth]{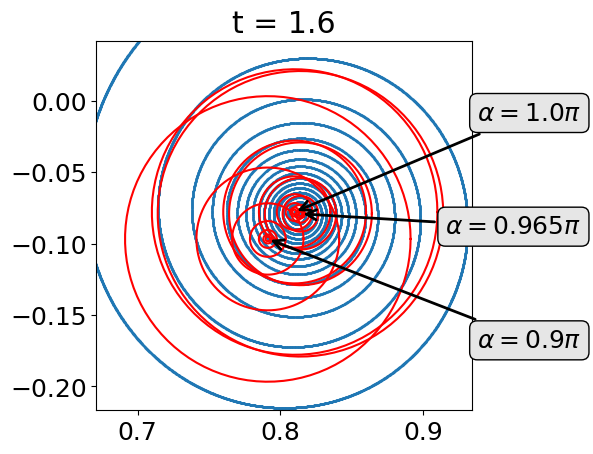}
\caption{close-up view}
\end{subfigure}
\caption{Loaded-wing configuration: Tracking of individual points along the vortex sheet with parameter-values $\alpha = 0.9\pi$, $0.965\pi$, $\pi$ at times $t=0.8$ and $t=1.6$ with regularization $\epsilon = 0.025$ and $N=100'001$ vortices. Red circles indicate the radial values for which the local vorticity maximal function have been evaluated. (left) global view of relative positions, (right) close-up view of positions}
\label{fig:localtracking_lw}
\end{figure}

The results of the local tracking presented in Figure \ref{fig:vortmaxlocal_lw} shows that the \emph{local} vorticity maximal function may experience larger fluctuations over time. As seen in Figure \ref{fig:vortmaxlocal_lw} (left column), the {local} vorticity maximal function can deteriorate at larger scales $r \approx r_{\mathrm{max}}$. This is explained as follows: The point at $\alpha = 0.9\pi$ is initially well-separated from the vortex sheet tips and hence initially displays a vorticity decay $\int_{B_r(z)} d|\omega_0| \lesssim r$, rather than the worse $r^{1/2}$ decay at the tips. At later times, the vortex sheet rolls up, bringing the point at $\alpha = 0.9\pi$ closer in the vicinity of the vortex sheet tip, so that the sudden increase in the local vorticity maximal function stems from the fact that the vortex sheet tip becomes ``visible'' to the local vorticity maximal function from that point on. In contrast, a local concentration (at small scales $r$) is not visible for any of the tracked points. In particular, the points near the vortex sheet tip ($\alpha = 0.965\pi$ and $\alpha = 1$) exhibit a scaling of order $\sim r^{1/2}$, uniform in time and for all values of $\epsilon$ considered (cp. Figure \ref{fig:vortmaxlocal_lw} central and right column).

Thus, the quantitative analysis of the simulation for the loaded wing configuration suggests a uniform decay of the vorticity maximal function, and hence strong convergence to a limiting energy-conservative weak solution.

\begin{figure}[H]
\centering
\begin{subfigure}{.32\textwidth}
\includegraphics[width=\textwidth]{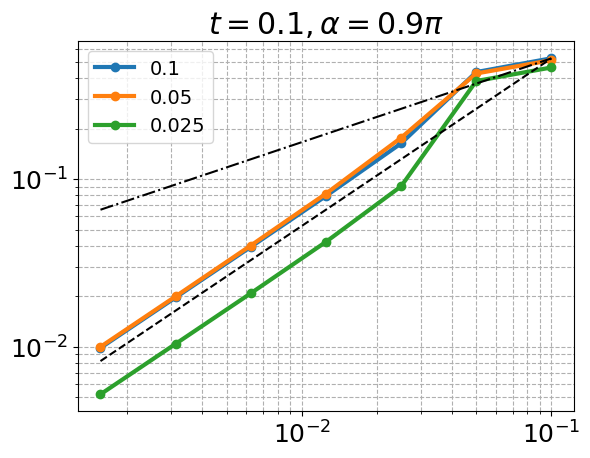}
\caption{$\alpha = 0.9\pi$}
\end{subfigure} 
\begin{subfigure}{.32\textwidth}
\includegraphics[width=\textwidth]{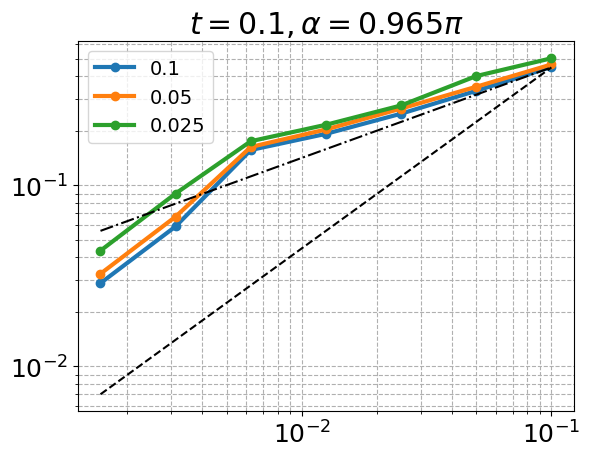}
\caption{$\alpha = 0.965\pi$}
\end{subfigure}
\begin{subfigure}{.32\textwidth}
\includegraphics[width=\textwidth]{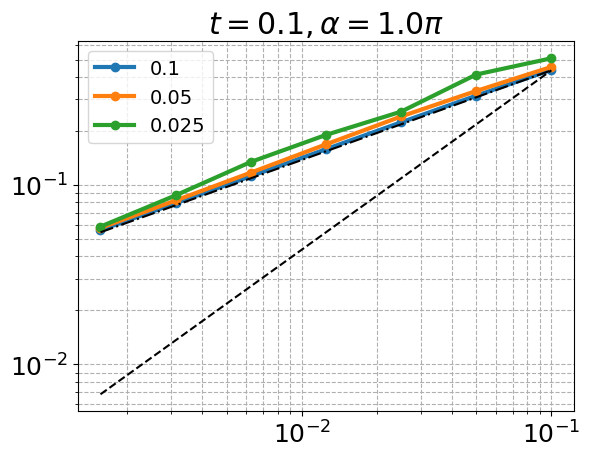}
\caption{$\alpha = \pi$}
\end{subfigure} 
\\[5pt]
\begin{subfigure}{.32\textwidth}
\includegraphics[width=\textwidth]{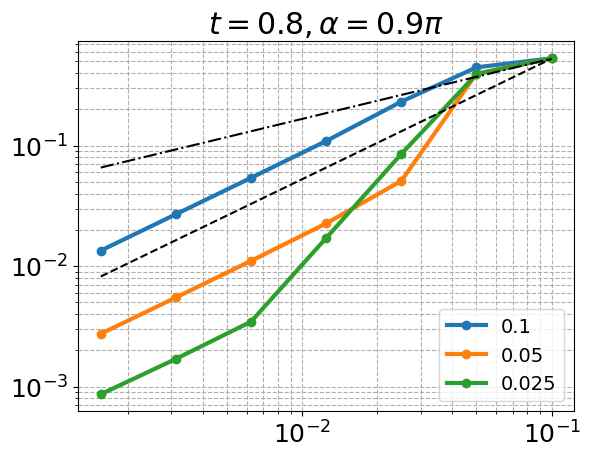}
\caption{$\alpha = 0.9\pi$}
\end{subfigure} 
\begin{subfigure}{.32\textwidth}
\includegraphics[width=\textwidth]{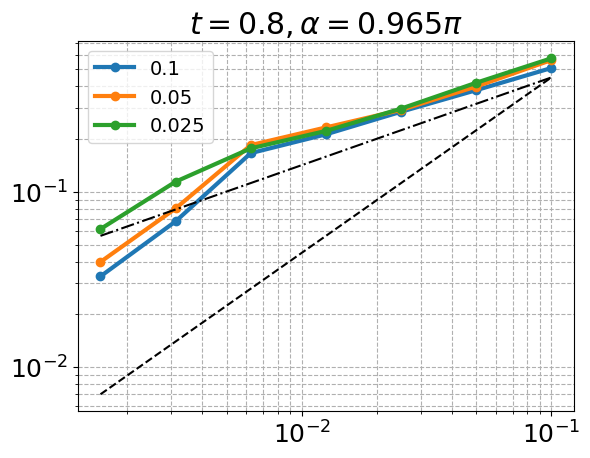}
\caption{$\alpha = 0.965\pi$}
\end{subfigure}
\begin{subfigure}{.32\textwidth}
\includegraphics[width=\textwidth]{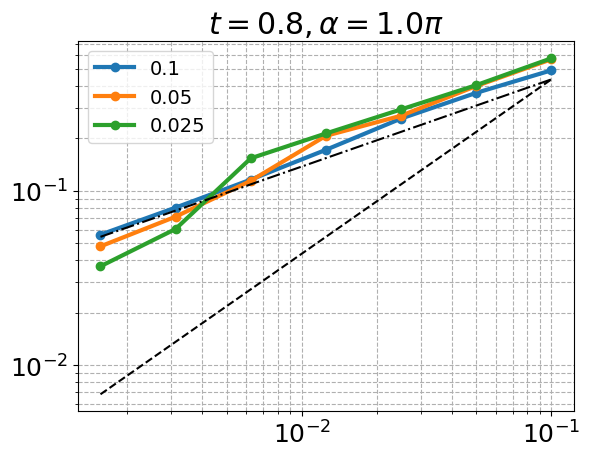}
\caption{$\alpha = \pi$}
\end{subfigure} 
\\[5pt]
\begin{subfigure}{.32\textwidth}
\includegraphics[width=\textwidth]{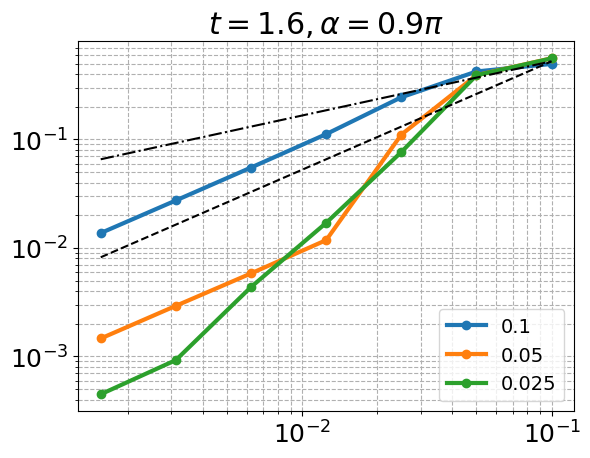}
\caption{$\alpha = 0.9\pi$}
\end{subfigure} 
\begin{subfigure}{.32\textwidth}
\includegraphics[width=\textwidth]{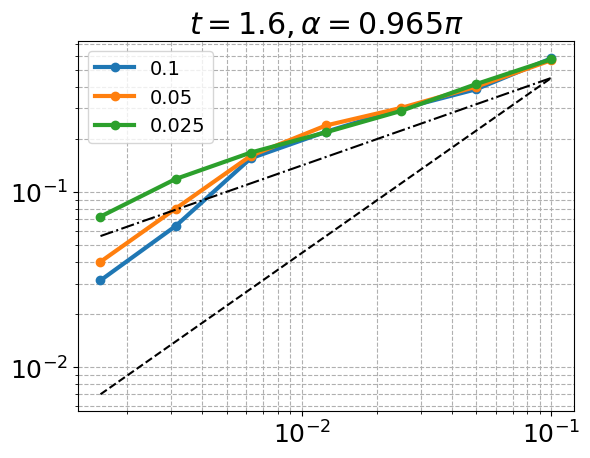}
\caption{$\alpha = 0.965\pi$}
\end{subfigure}
\begin{subfigure}{.32\textwidth}
\includegraphics[width=\textwidth]{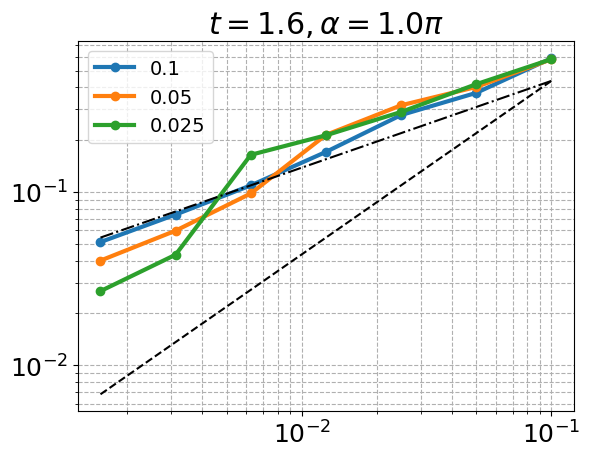}
\caption{$\alpha = \pi$}
\end{subfigure} 
\caption{Loaded-wing configuration: Evaluation of the local vorticity maximal function as a function of radius $r$, at parameter values $\alpha = 0.9\pi$ (left), $0.965\pi$ (center) and $\pi$ (right). Evaluated for regularization parameters $\epsilon = 0.1$, $0.05$ and $0.025$, and at times $t = 0.1$ (top row), $t=0.8$ (middle row), $t= 1.6$ (bottom row). Also indicated are algebraic decay of order $r^{1/2}$ (black dashed line) and $r$ (black solid line).}
\label{fig:vortmaxlocal_lw}
\end{figure}

\subsection{Fuselage flap configuration} 

After our discussion of the elliptically loaded wing, we repeat our numerical experiments for the fuselage flap configuration, described in detail in Section \ref{sec:numerical:setup}. We have computed the evolution of the vortex sheet up to time $t=8.0$, using different values of the regularisation parameter $\epsilon = 0.1$, $0.05$ and $0.025$. For each regularisation, we again initialize the vortex sheet with a total of $N=100'001$ discrete vortices. A constant time-step of $\Delta t = 0.005$ has been chosen for these simulations. For this configuration, we can qualitatively distinguish between an early evolution ($t\le 1.0$), and a later stage ($t \ge 2$). The quantitative analysis of our simulation results will thus be carried out first for early times (qualitative similarity with loaded wing roll-up) and then for later times (complicated interaction of negative and positive vortex spirals). 

\subsubsection{Early evolution}

We first focus on the qualitative evolution of the vortex sheet for early times. Representative plots of the vortex sheet evolution with different regularization parameters $\epsilon = 0.1$, $0.025$ are shown in Figure \ref{fig:earlyevo}.

\begin{figure}[H]
\centering
\begin{subfigure}{.45\textwidth}
\includegraphics[width=\textwidth]{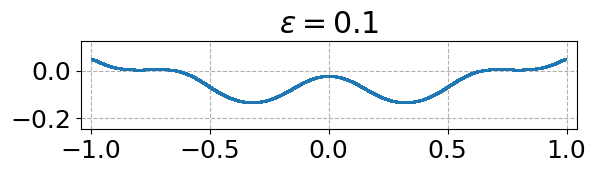}
\caption{$t = 0.1$}
\end{subfigure}
\begin{subfigure}{.45\textwidth}
\includegraphics[width=\textwidth]{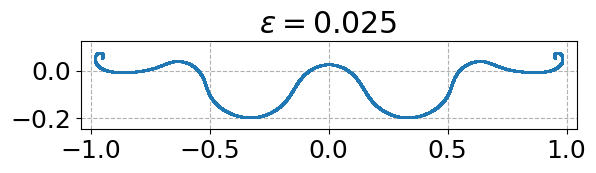}
\caption{$t = 0.1$}
\end{subfigure} 
\\[5pt]
\begin{subfigure}{.45\textwidth}
\includegraphics[width=\textwidth]{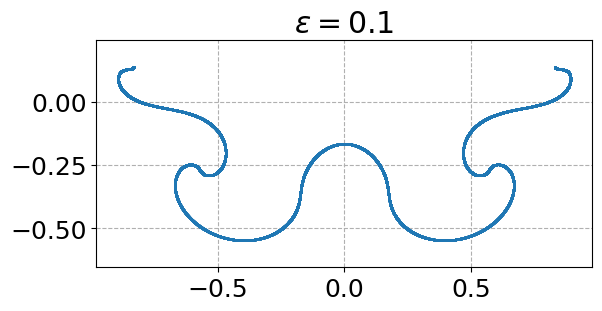}
\caption{$t = 0.5$}
\end{subfigure}
\begin{subfigure}{.45\textwidth}
\includegraphics[width=\textwidth]{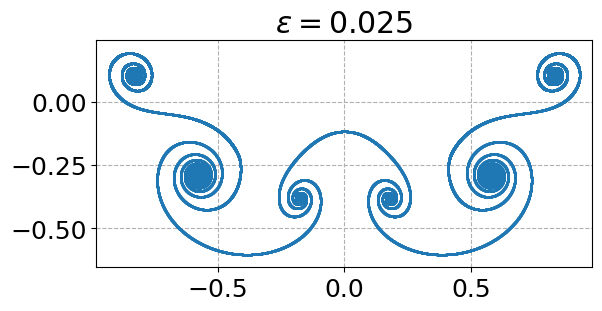}
\caption{$t = 0.5$}
\end{subfigure} 
\\[5pt]
\begin{subfigure}{.45\textwidth}
\includegraphics[width=\textwidth]{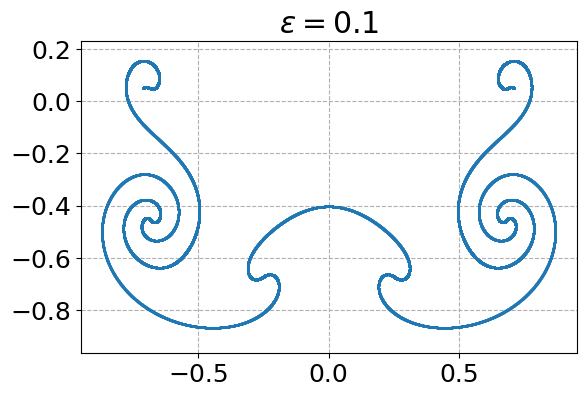}
\caption{$t = 1.0$}
\end{subfigure}
\begin{subfigure}{.45\textwidth}
\includegraphics[width=\textwidth]{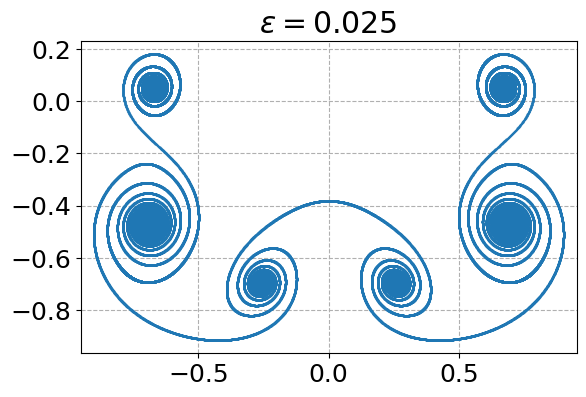}
\caption{$t = 1.0$}
\end{subfigure} 
\caption{Fuselage flap configuration: Evolution of vortex sheet with regularization parameter $\epsilon = 0.1$ (left) and $\epsilon = 0.025$ (right) for $t = 0.1$, $0.5$, $1.0$. Computed with the vortex-blob scheme using $N=100'001$ vortices.}
\label{fig:earlyevo}
\end{figure}

Figure \ref{fig:earlyevo} reproduces the evolution depicted in \cite[Figure 17]{Krasny1987}. Again, in these plots, each vortex is represented as a single dot. The appearance of a continuous curve is due to the density of vortex points. As is clearly visible in Figure \ref{fig:earlyevo}, a smaller value of the regularization parameter leads to a stronger roll-up of the vortex sheet. Comparing Figure \ref{fig:evolution_lw} (loaded wing) and Figure \ref{fig:earlyevo} (fuselage flap), the roll-up in spirals appears visually to be quite similar, so that we expect to find a similar decay of the vorticity maximal function for both the fuselage flap and loaded wing configurations, at these early times.

The numerically computed evolution of the discrete vorticity maximal function $M_r(\omega_\epsilon(t))$ for different values of $r$ and at different times $t$ is shown in Figure \ref{fig:vortearly}, using the global binning approach explained in Section \ref{sec:global}. The results at early times $t=0$, $0.1$, $0.5$ and $1.0$ are shown in Figure \ref{fig:vortearly}.

As expected, the results presented in Figure \ref{fig:vortearly} show that the vorticity maximal function exhibits a very similar uniform decay of the vorticity maximal function at early times in the evolution of the fuselage flap configuration, compared with the loaded wing configuration. With an initial decay $\sim r^{1/2}$ at $t=0$, this decay persists up to time $t=1$. As evidenced by Figure \ref{fig:vortearly}, this decay is also uniform over the considered regularization parameters $\epsilon = 0.1$, $0.05$, $0.025$.

\begin{figure}[H]
\centering
\begin{subfigure}{.45\textwidth}
\includegraphics[width=\textwidth]{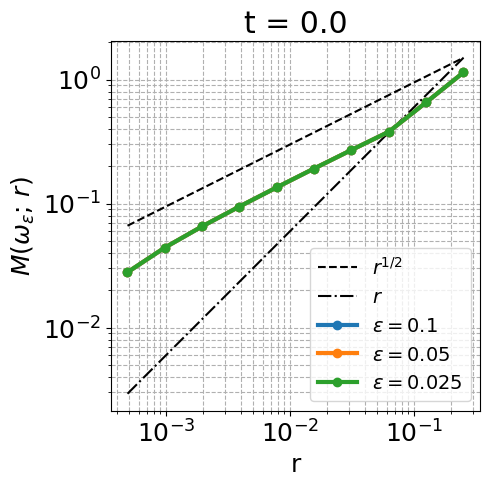}
\caption{$t = 0.0$}
\end{subfigure}
\begin{subfigure}{.45\textwidth}
\includegraphics[width=\textwidth]{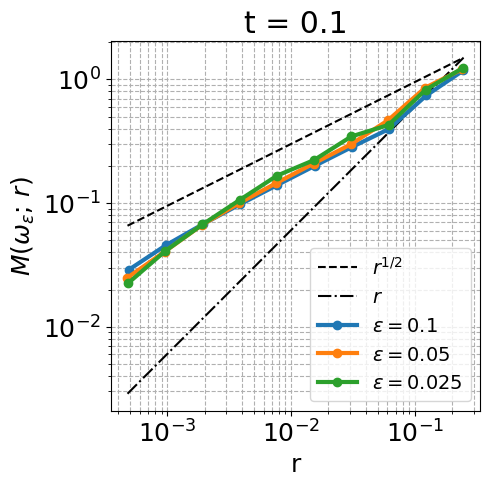}
\caption{$t = 0.1$}
\end{subfigure} 
\\[5pt]
\begin{subfigure}{.45\textwidth}
\includegraphics[width=\textwidth]{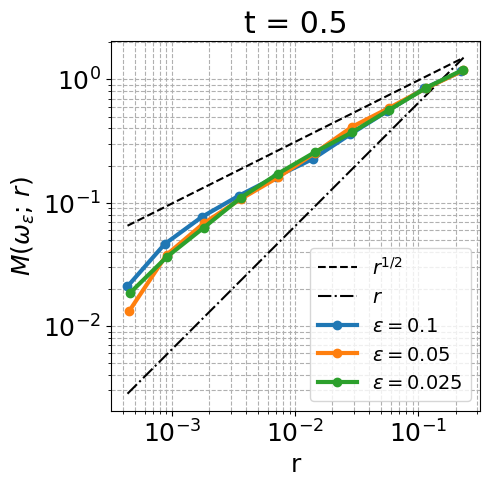}
\caption{$t = 0.5$}
\end{subfigure}
\begin{subfigure}{.45\textwidth}
\includegraphics[width=\textwidth]{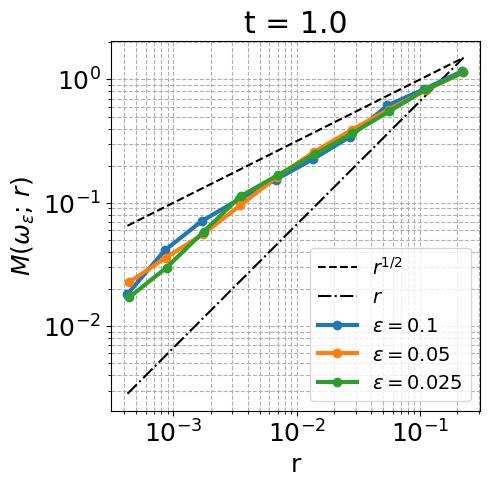}
\caption{$t = 1.0$}
\end{subfigure} 
\caption{Fuselage flap configuration: Evolution of vorticity maximal function $M_r(\omega_\epsilon(t))$ with regularization parameter $\epsilon \in \{ 0.025,0.05,0.1 \}$ at $t = 0.0$, $0.1$, $0.5$, $1.0$. Computed with the vortex-blob scheme using $N=100'001$ vortices.}
\label{fig:vortearly}
\end{figure}

Our main interest is thus in the evolution of this fuselage flap configuration at later times, when the spirals of negative and positive vorticity begin to strongly interact: Will the vorticity maximal function cease to show a uniform decay during these strong interactions?

\begin{figure}[H]
\centering
\begin{subfigure}{.45\textwidth}
\includegraphics[width=\textwidth]{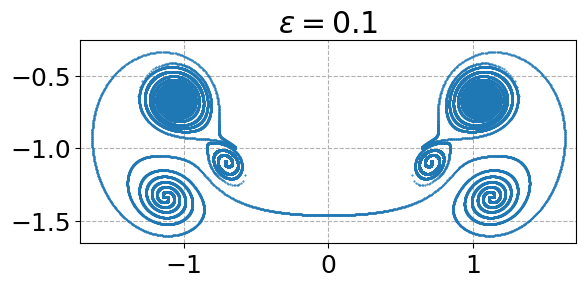}
\caption{$t = 4$}
\end{subfigure}
\begin{subfigure}{.45\textwidth}
\includegraphics[width=\textwidth]{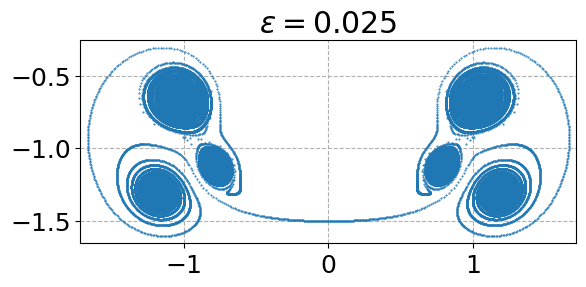}
\caption{$t = 4$}
\end{subfigure} 
\\[5pt]
\begin{subfigure}{.45\textwidth}
\includegraphics[width=\textwidth]{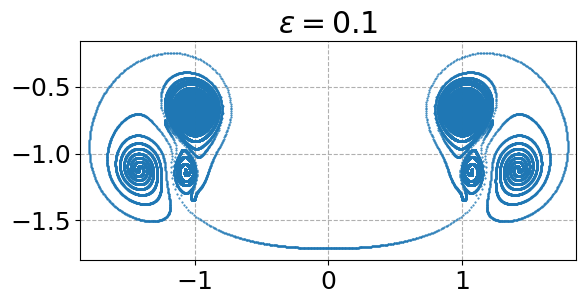}
\caption{$t = 5$}
\end{subfigure}
\begin{subfigure}{.45\textwidth}
\includegraphics[width=\textwidth]{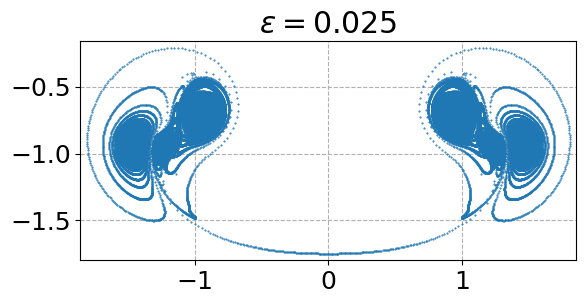}
\caption{$t = 5$}
\end{subfigure} 
\\[5pt]
\begin{subfigure}{.45\textwidth}
\includegraphics[width=\textwidth]{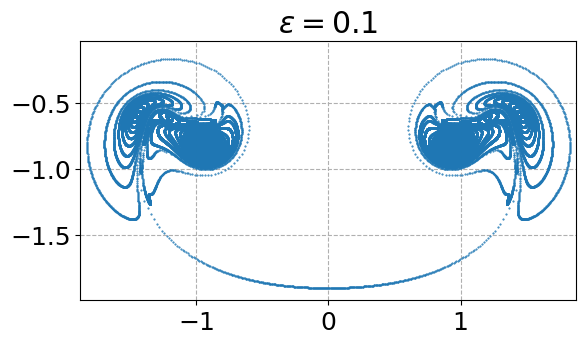}
\caption{$t = 6$}
\end{subfigure}
\begin{subfigure}{.45\textwidth}
\includegraphics[width=\textwidth]{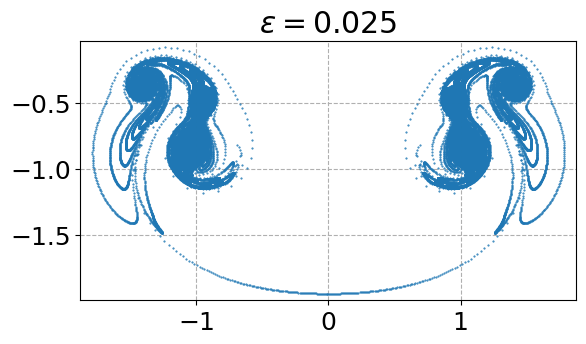}
\caption{$t = 6$}
\end{subfigure}
\\[5pt]
\begin{subfigure}{.45\textwidth}
\includegraphics[width=\textwidth]{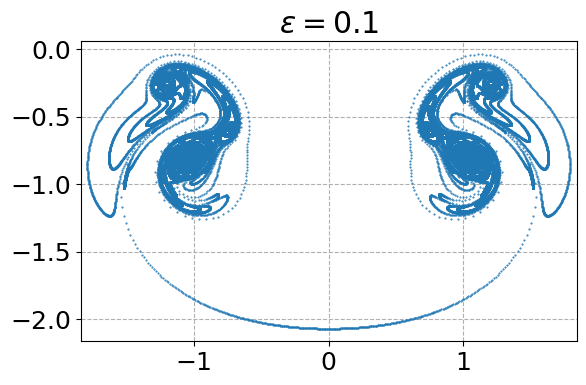}
\caption{$t = 7$}
\end{subfigure}
\begin{subfigure}{.45\textwidth}
\includegraphics[width=\textwidth]{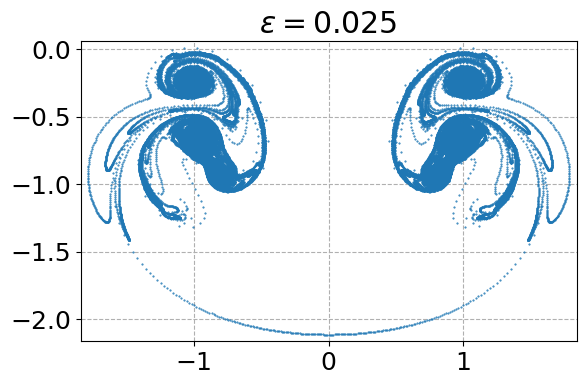}
\caption{$t = 7$}
\end{subfigure}
\\[5pt]
\begin{subfigure}{.45\textwidth}
\includegraphics[width=\textwidth]{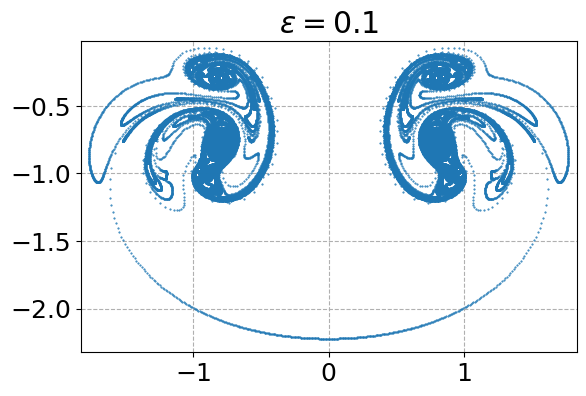}
\caption{$t = 8$}
\end{subfigure}
\begin{subfigure}{.45\textwidth}
\includegraphics[width=\textwidth]{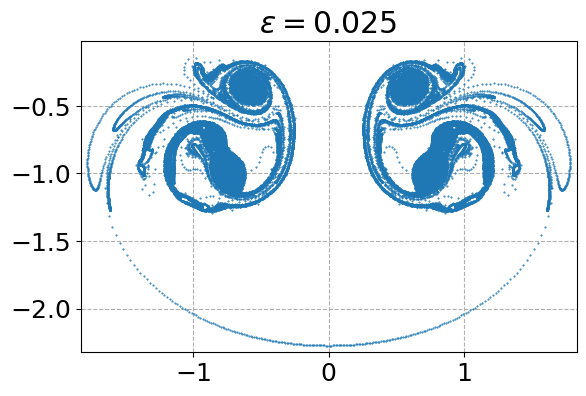}
\caption{$t = 8$}
\end{subfigure} 
\caption{Fuselage flap configuration: Evolution of vortex sheet with regularization parameter $\epsilon = 0.1$ (left) and $\epsilon = 0.025$ (right) for $t \in [4,8]$. Computed with the vortex-blob scheme using $N=100'001$ vortices.}
\label{fig:lateevo}
\end{figure}

\subsubsection{Later evolution}

A visual illustration of the evolution of the fuselage flap configuration at later times $t \in [4,8]$ is presented in Figure \ref{fig:lateevo} (cp. \cite[Figure 23]{Krasny1987}). Our simulation reproduces the very complex interaction of the positive and negative parts of the vortex sheet at these late times, first observed in \cite{Krasny1987}. 

As mentioned in the introduction, this visual indication of the formation of small scale features in vortex sheets with varying sign of the vorticity might indicate that the velocity corresponding to such unsigned vortex sheets (e.g. fuselage flap) could be less regular than the velocity of their signed counterparts (e.g. loaded wing). The availability of an a priori bound on the vorticity maximal function, 
\begin{align} \label{eq:logbound}
\int_{B_r(x)} d|\omega(y)| \equiv \int_{B_r(x)} d\omega(y) \le C |\log(r)|^{-1/2},
\end{align}
which is only available in the signed case $\omega \ge 0$, might be interpreted as circumstantial evidence in support of the higher regularity of the signed case than of the unsigned case. Numerical computations such as the ones presented for the loaded wing configuration in fact appear to exhibit much more regularity,
\[
\int_{B_r(x)} d|\omega(y)| \le C r^\alpha, \quad \alpha > 0.
\]
This observed algebraic decay is clearly not explained by the logarithmic bound \eqref{eq:logbound}, even for non-negative vorticity. In particular, the persistence of such an algebraic bound on the vorticity maximal function may rely on very different, as of yet not understood, dynamical properties of the vortex sheet evolution, which are not explained by the uniform $H^{-1}$-bound which leads to \eqref{eq:logbound}.

To see the effect of the complex interaction between positive and negative vortex spiral on the regularity properties of the associated flow, we again compute the vorticity maximal function numerically, using the algorithm described in Section \ref{sec:global}. The results for times $t\in [2,4]$ are shown in Figure \ref{fig:latemax}. Astonishingly, this numerical analysis of the regularity based on the vorticity maximal function does not provide any support in favour of any deterioration of the vortex sheet even during these complex interactions. Instead, at the level of the vorticity maximal functions we observe a persistent uniform decay, despite the apparently much more complex behaviour of the vortex sheet at these times. Thus, also at these late times do our computation strongly suggest strong convergence (up to a subsequence) to a limiting, energy-conservative weak solution (cp. Theorem \ref{thm:strongconv}) in the limit $\epsilon \to 0$.

\begin{figure}
\centering
\begin{subfigure}{.45\textwidth}
\includegraphics[width=\textwidth]{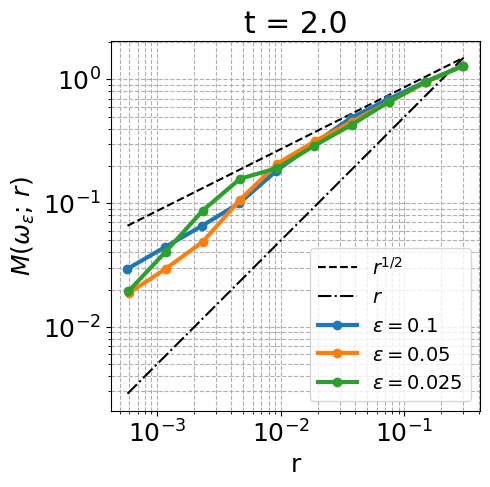}
\caption{$t = 2.0$}
\end{subfigure}
\begin{subfigure}{.45\textwidth}
\includegraphics[width=\textwidth]{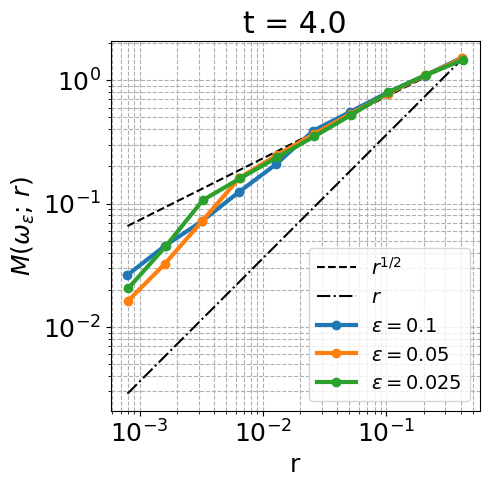}
\caption{$t = 4.0$}
\end{subfigure} 
\\[5pt]
\begin{subfigure}{.45\textwidth}
\includegraphics[width=\textwidth]{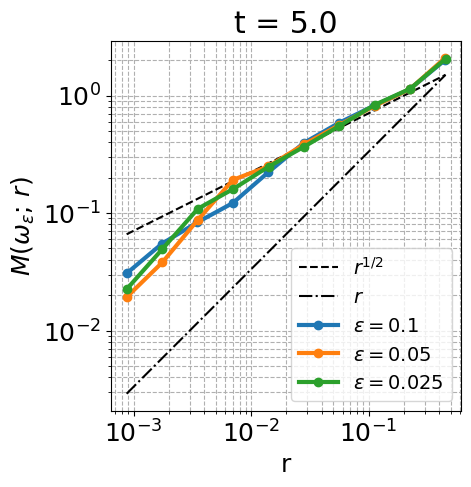}
\caption{$t = 5.0$}
\end{subfigure}
\begin{subfigure}{.45\textwidth}
\includegraphics[width=\textwidth]{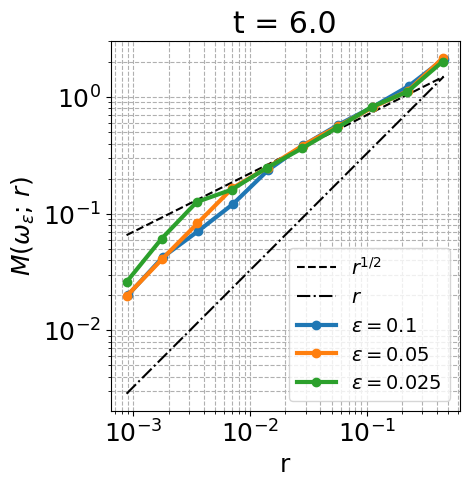}
\caption{$t = 6.0$}
\end{subfigure} 
\\[5pt]
\begin{subfigure}{.45\textwidth}
\includegraphics[width=\textwidth]{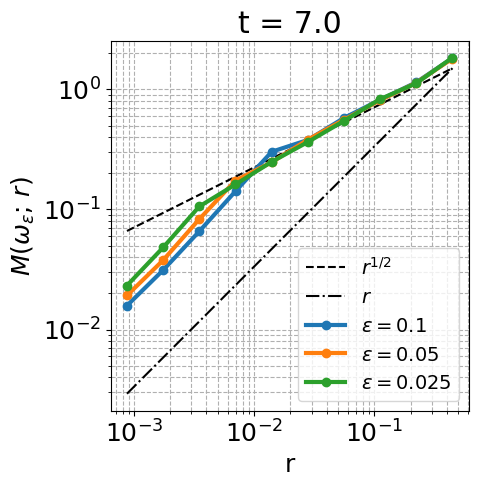}
\caption{$t = 7.0$}
\end{subfigure}
\begin{subfigure}{.45\textwidth}
\includegraphics[width=\textwidth]{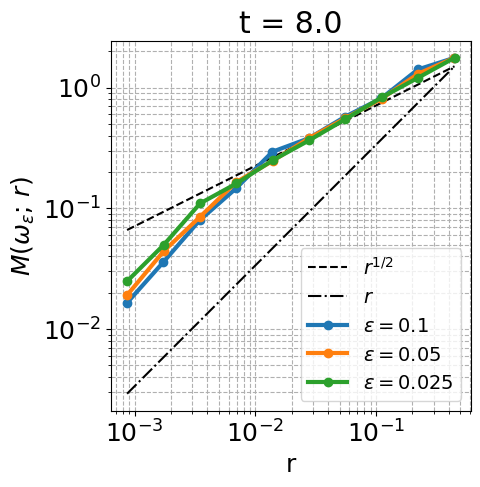}
\caption{$t = 8.0$}
\end{subfigure} 
\caption{Fuselage flap configuration:  Evolution of vorticity maximal function $M_r(\omega_\epsilon(t))$ with regularization parameter $\epsilon \in \{ 0.025,0.05,0.1 \}$ for $t \in [2,8]$. Computed with the vortex-blob scheme using $N=100'001$ vortices.}
\label{fig:latemax}
\end{figure}

\subsubsection{Local behaviour}
Finally, to confirm this observation of uniform decay of the vorticity maximal function, and in order to get a more intuitive explanation why this uniform decay persists despite the complex interactions, we track the local vorticity concentration, following individual points on the vortex sheet. Taking the approach described in Section \ref{sec:local}, we follow three points with parameters $\alpha = 0.675\pi$, $0.81\pi$ and $\pi$ and compute the evolution of the local vorticity concentration
\begin{align} \label{eq:localconc}
r \mapsto \int_{B_r(z(\alpha,t))} d|\omega|.
\end{align}
The location of the tracked points on the vortex sheet at early times is depicted in Figure \ref{fig:earlypos}. The point at $\alpha=0.675\pi$ corresponds to the centre of one of the vortex spirals undergoing the complex interactions. $\alpha = \pi$ tracks the tip of the vortex sheet which experiences roll-up even in the signed loaded wing case. Finally, $\alpha = 0.81\pi$ has been chosen as a representative point which will evolve at the interface between the spirals of positive and negative vorticity surrounding $\alpha = 0.675\pi$ and $\alpha = \pi$. 

\begin{figure}[H]
\centering
\includegraphics[width=0.5\textwidth]{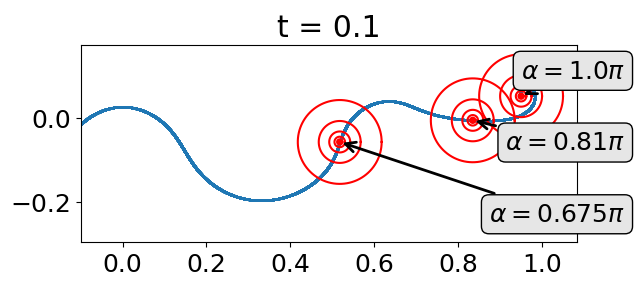}
\caption{Fuselage flap configuration: Points along the vortex sheet chosen for the local tracking, at time $t = 0.1$}
\label{fig:earlypos}
\end{figure}

\begin{figure}[H]
\centering
\begin{subfigure}{.32\textwidth}
\includegraphics[width=\textwidth]{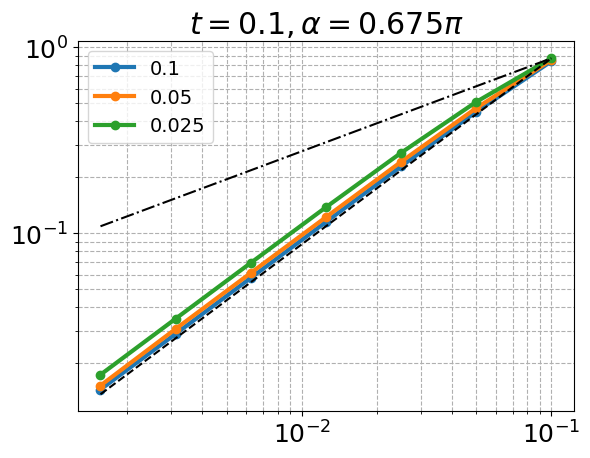}
\caption{$\alpha = 0.675\pi$}
\end{subfigure} 
\begin{subfigure}{.32\textwidth}
\includegraphics[width=\textwidth]{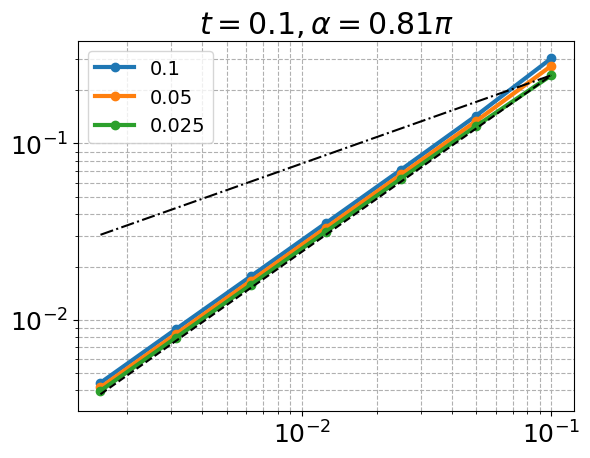}
\caption{$\alpha = 0.81\pi$}
\end{subfigure}
\begin{subfigure}{.32\textwidth}
\includegraphics[width=\textwidth]{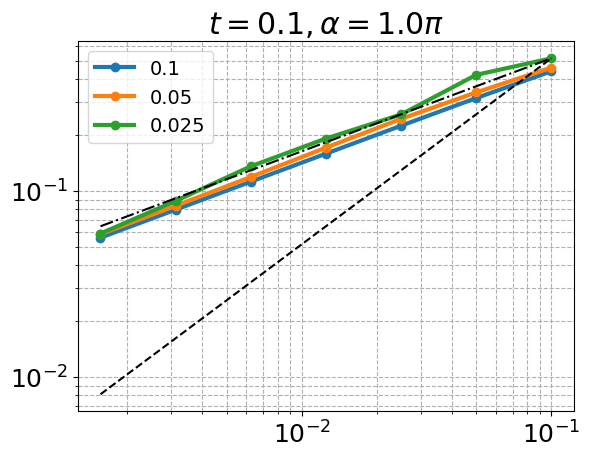}
\caption{$\alpha = \pi$}
\end{subfigure} 
\caption{Fuselage flap configuration: Initial decay of the local vorticity concentration at the tracked points (cp. Figure \ref{fig:earlypos}), at $t=0.1$. Evaluated for regularization parameters $\epsilon = 0.1$, $0.5$ and $0.025$, and at times $t = 0.1$ (top row), $t=0.8$ (middle row), $t= 1.6$ (bottom row). Also indicated are algebraic decay of order $r^{1/2}$ (black dashed line) and $r$ (black solid line). }
\label{fig:trackdecay}
\end{figure}

As shown in Figure \ref{fig:trackdecay}, the initial decay of the local vorticity concentration \eqref{eq:localconc} is $\sim r$ at $\alpha = 0.675\pi$, $0.81\pi$, and $\sim r^{1/2}$ at the vortex sheet tip $\alpha = \pi$. The subsequent evolution of the vortex sheet and the tracked points during the complex interaction of the vortex sheet is depicted in Figure \ref{fig:visualtracking}. Visible are the complicated intertwining of different parts of the vortex sheet, which had already been pointed out in \cite{Krasny1987}. A quantitative analysis of the corresponding evolution of the local vorticity function at the tracked points is provided in Figure \ref{fig:trackingdecay}. The results of the local tracking show very clearly that the vorticity concentration is highest at the centers of the individual spirals, corresponding to regions of either positive or negative vorticity. The vorticity concentration in between these regions is observed to be much reduced in the region of direct ``contact'' between the spirals, i.e. the region along which the point $z(\alpha,t)$ at $\alpha = 0.81\pi$ evolves (cp. Figure \ref{fig:trackingdecay}, center column). As the centers of the spirals with the highest concentrations of positive and negative vorticities do not collide, but rather circle each other throughout the evolution, the vortex concentration at their core does not appear to be affected by the interaction with the parts of the vortex sheet with opposite sign. The consequence of this dynamical behaviour, and the non-collision of the vortex-spiral centers, is a decay of the vorticity maximal function, behaving akin to what has been observed in the case of a vortex sheet with effectively distinguished sign (and mirror symmetry), in Section \ref{sec:loadedwing}. 

To summarize: also for this fuselage flap configuration, an unsigned vortex sheet case, do we find strong numerical evidence in favour of a uniform decay of the vorticity maximal function, and hence strong convergence to an energy-conservative weak solution by Theorem \ref{thm:strongconv}. In particular, at the considered regularization parameters, no indication of any concentration phenomena as described in \cite{DipernaMajda1987a,DipernaMajda1987b} is found.

\begin{figure}[H]
\centering
\begin{subfigure}{.49\textwidth}
\includegraphics[width=\textwidth]{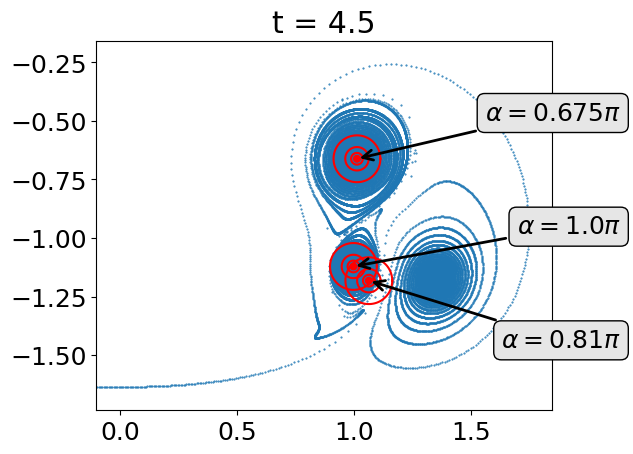}
\caption{$t = 4.5$}
\end{subfigure} 
\begin{subfigure}{.49\textwidth}
\includegraphics[width=.75\textwidth]{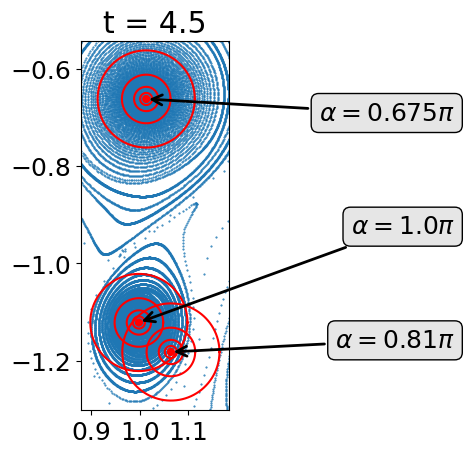}
\caption{close-up view}
\end{subfigure}
\\[5pt]
\begin{subfigure}{.49\textwidth}
\includegraphics[width=\textwidth]{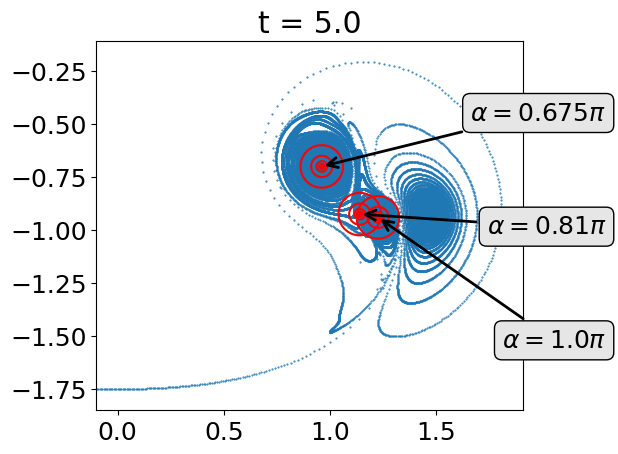}
\caption{$t = 5.0$}
\end{subfigure} 
\begin{subfigure}{.49\textwidth}
\includegraphics[width=.95\textwidth]{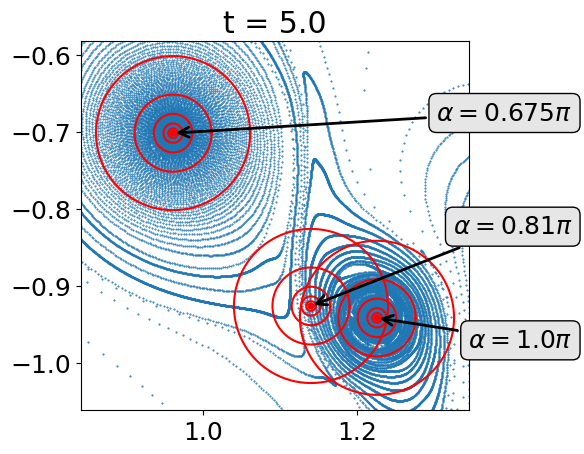}
\caption{close-up view}
\end{subfigure}
\\[5pt]
\begin{subfigure}{.49\textwidth}
\includegraphics[width=\textwidth]{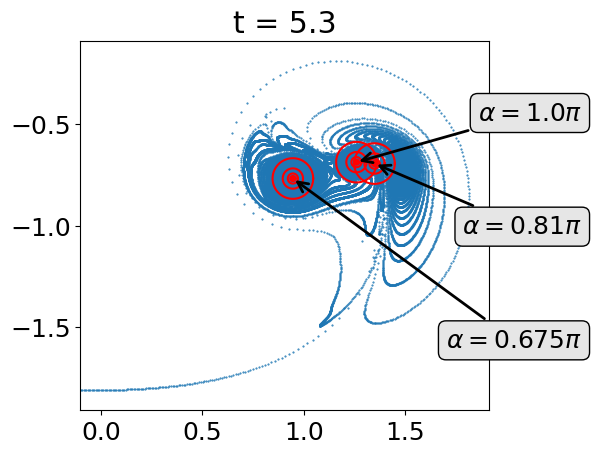}
\caption{$t=5.3$}
\end{subfigure} 
\begin{subfigure}{.49\textwidth}
\includegraphics[width=\textwidth]{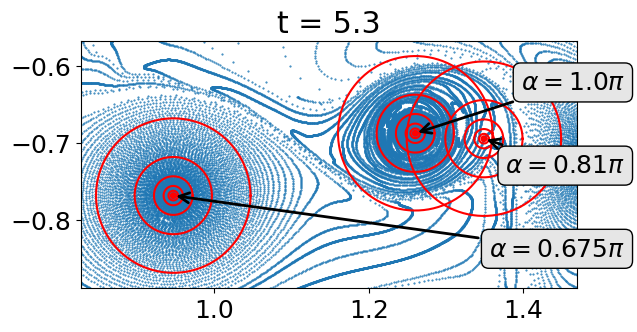}
\caption{close-up view}
\end{subfigure}
\caption{Fuselage flap configuration: Visual tracking of the evolution of individual points at $\alpha = 0.675\pi$, $0.81\pi$ and $\pi$, along the vortex sheet. Red circles indicate the length scales for which the local vorticity has been evaluated. (left) general view; (right) close-up view}
\label{fig:visualtracking}
\end{figure}

\begin{figure}[H]
\centering
\begin{subfigure}{.32\textwidth}
\includegraphics[width=\textwidth]{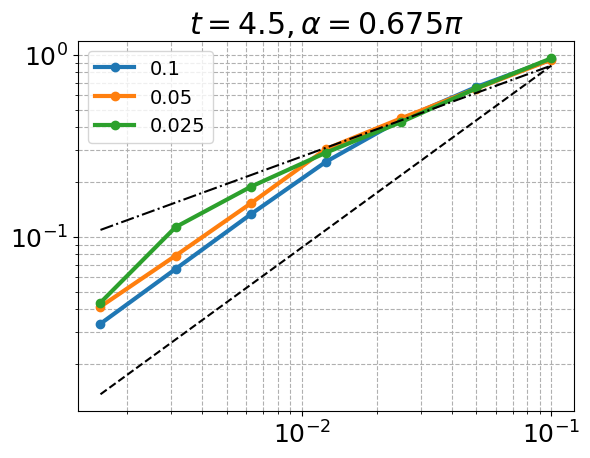}
\caption{$\alpha = 0.675\pi$}
\end{subfigure} 
\begin{subfigure}{.32\textwidth}
\includegraphics[width=\textwidth]{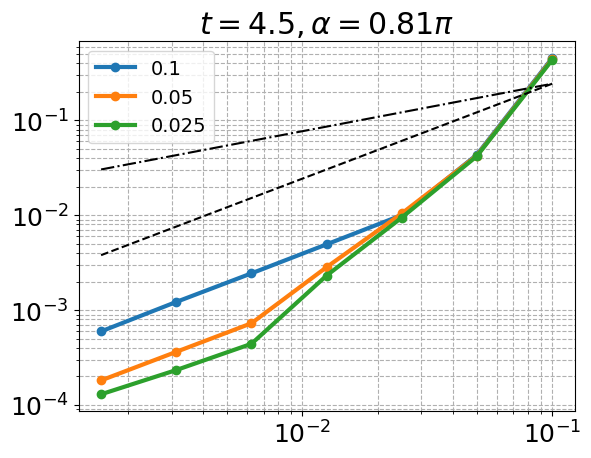}
\caption{$\alpha = 0.81\pi$}
\end{subfigure}
\begin{subfigure}{.32\textwidth}
\includegraphics[width=\textwidth]{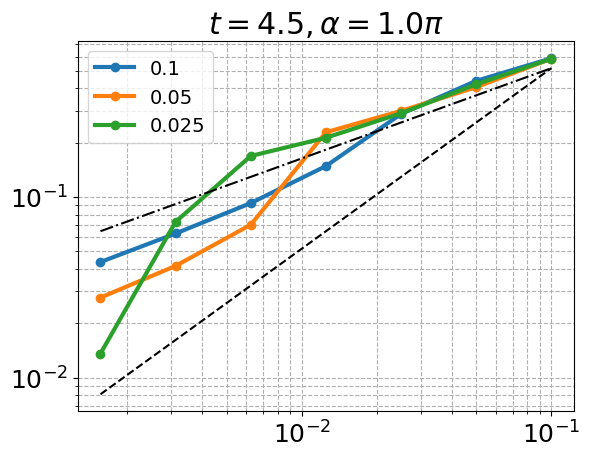}
\caption{$\alpha = \pi$}
\end{subfigure} 
\\[5pt]
\begin{subfigure}{.32\textwidth}
\includegraphics[width=\textwidth]{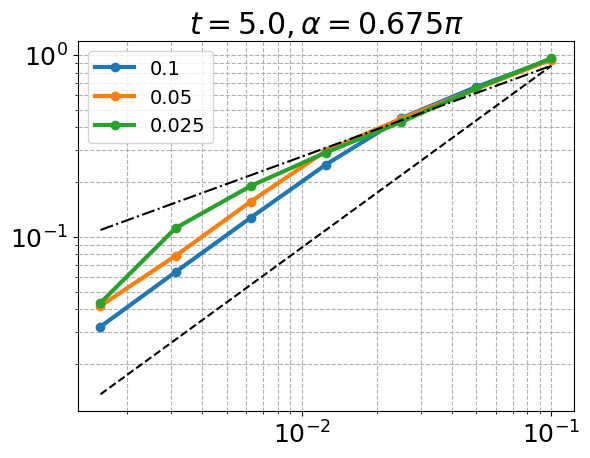}
\caption{$\alpha = 0.675\pi$}
\end{subfigure} 
\begin{subfigure}{.32\textwidth}
\includegraphics[width=\textwidth]{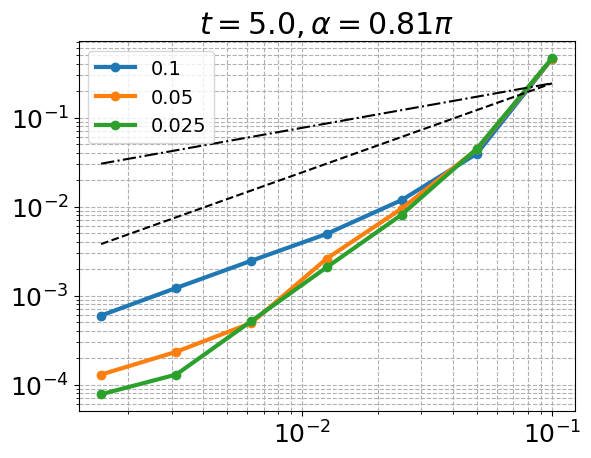}
\caption{$\alpha = 0.81\pi$}
\end{subfigure}
\begin{subfigure}{.32\textwidth}
\includegraphics[width=\textwidth]{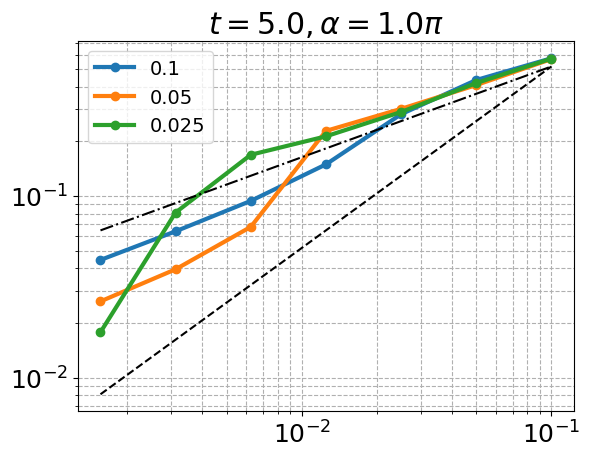}
\caption{$\alpha = \pi$}
\end{subfigure} 
\\[5pt]
\begin{subfigure}{.32\textwidth}
\includegraphics[width=\textwidth]{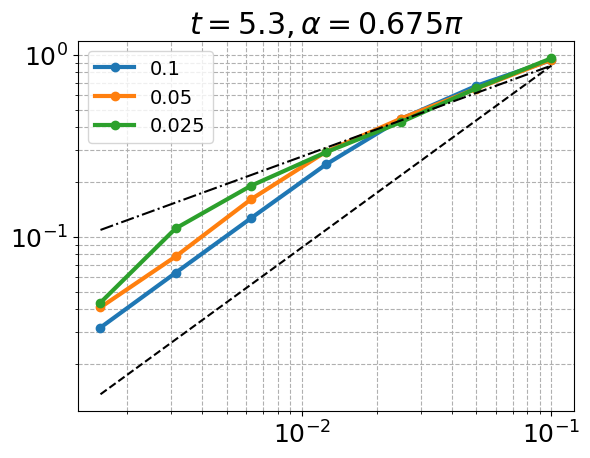}
\caption{$\alpha = 0.675\pi$}
\end{subfigure} 
\begin{subfigure}{.32\textwidth}
\includegraphics[width=\textwidth]{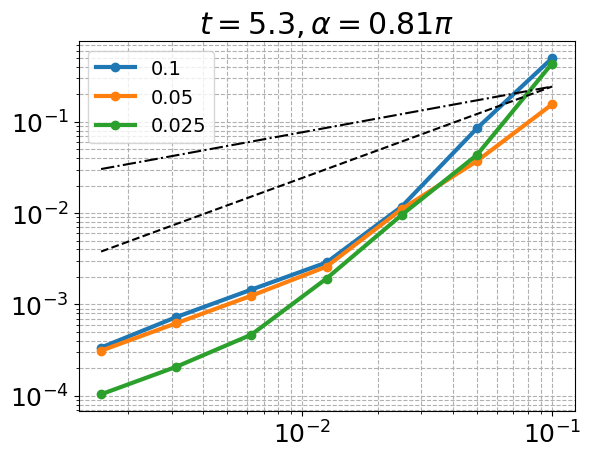}
\caption{$\alpha = 0.81\pi$}
\end{subfigure}
\begin{subfigure}{.32\textwidth}
\includegraphics[width=\textwidth]{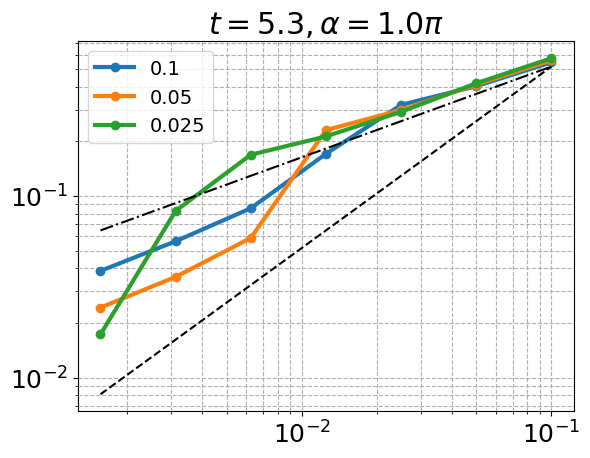}
\caption{$\alpha = \pi$}
\end{subfigure} 
\caption{Fuselage flap configuration: Local decay of the vorticity along tracked points, at times $t = 4.5$ (top row), $t=5.0$ (middle row), $t= 5.3$ (bottom row). Evaluated for regularization parameters $\epsilon = 0.1$, $0.5$ and $0.025$. Also indicated are algebraic decay of order $r^{1/2}$ (black dashed line) and $r$ (black solid line).}
\label{fig:trackingdecay}
\end{figure}

\section{Conclusions} \label{sec:conclusion}

The present work considers the question of energy concentration in the evolution of vortex sheets in two dimensions. While some results on the existence of weak solutions and convergence of numerical approximations are available in the signed vortex sheet case $\omega \ge 0$, obtaining a more precise understanding of the behaviour of (approximate) solutions of the incompressible Euler equations with vortex sheet initial data, without any sign restriction, remains an outstanding challenge. The detailed simulations performed in \cite{Krasny1987} have been influential in demonstrating the apparently much more complex dynamics of unsigned vortex sheets in comparison with their signed counterparts. Based on our theoretical understanding of these equations, it can at present not be ruled out that in the case of unsigned vortex sheets, approximate solution sequences might exhibit concentration phenomena in the limit. In fact, the computations presented in \cite{Krasny1987} have been proposed as possible evidence of concentration effects in approximate solution sequences with unsigned vortex sheet initial data \cite[Remark 3.2]{DipernaMajda1987a}.

The present work aims to shed some new light on the long-standing question of concentrations in approximate solution sequences. To this end, analytical tools are first developed to allow a quantitative analysis of the numerical results: The structure functions, already studied in \cite{LMP2020} form the basis of this analysis. As explained in the present work (cp. Remark \ref{rem:concentration}), uniform bounds on the structure functions imply that no concentration in approximate solution sequences can take place in the limit. A novel expression for the structure function is derived, relating the decay of structure functions to the decay of average correlations of the vorticity. Based on this identity, a priori estimates on the structure function are given for solutions with a uniform decay of the vorticity maximal function. This also provides a different proof (and a slight improvement in terms of assumptions) of the logarithmic circulation theorem of Diperna and Majda \cite[Theorem 3.1]{DipernaMajda1987b} in the present setting. The latter a priori estimates are then applied to the vortex-blob method to prove that a uniform algebraic decay of the vorticity maximal function implies strong convergence of approximate solutions generated by the vortex-blob method, to an energy-conservative solution. 

The presented analytic estimates provide an efficient practical criterion to test for the convergence of solution sequences obtained in numerical experiments. This is illustrated with numerical experiments based on the vortex-blob method with the vortex sheet initial data first studied in \cite{Krasny1987}. The numerical experiments of \cite{Krasny1987} are revisited, and analysed in terms of the decay properties of the obtained vorticity maximal functions. The proposed estimates based on the vorticity maximal function are particularly suited to vortex-blob and vortex-point methods, since these methods provide direct access to the discretized vorticity $\omega_\epsilon$, a sum of individual vortices, rather than the velocity values on e.g. a regular grid, which would be required for the evaluation of the structure functions.

The main new finding of the present work is that, despite the evidently much more complex dynamics of vortex-sheet computations without a sign-restriction compared to simulations with vorticity of distinguished sign, the behaviour in terms of the temporal evolution of the vorticity maximal function provides strong numerical evidence that no concentration phenomena are observed for the considered initial data, even in the limit of zero regularisation $\epsilon \to 0$. In particular, as shown in the present work, such non-concentration would imply strong (subsequential) convergence of approximate solution sequences to an energy-conservative solution of the incompressible Euler equations, for both signed and unsigned vortex sheet initial data. The observed persistence of regularity in approximate solution sequences of the incompressible Euler equations, is similar to recent observations obtained from spectral methods in \cite{LMP2019}.

Despite the current lack of rigorous a priori estimates, which can explain the observed persistence of regularity in the two-dimensional incompressible Euler equations, the numerical results presented in this work suggest that there may be (as of yet undiscovered) dynamical mechanisms which prevent dynamical concentration effects. The analytic tools developed in the present work pave the way for further investigations of the detailed dynamics of unsigned vortex sheets. In future numerical experiments, it would be desirable to consider different initial configurations and probe smaller values of the regularization parameter $\epsilon$. 

\subsection*{Data availability statement}
The datasets generated during and/or analysed during the current study are available from the corresponding author on reasonable request.

\subsection*{Conflict of interest statement}

The author declares no competing interests.

\appendix


\section{Derivation of the structure function identity \eqref{eq:structvort}}
\label{app:structfun}

The goal of this section is to derive an expression for the structure function 
\[
S_2(u;r)^2 =  \int_{\dom} \fint_{B_r(0)} |u(x+h)-u(x)|^2 \, dh \, dx,
\]
in terms of the vorticity $\omega$ (in two spatial dimensions). Let us throughout assume that $u\in L^2_x$ is divergence-free, $\omega = \curl(u) \in L^1_x \cap L^\infty_x$ has compact support and $u = K\ast\omega$ (cp. Remark \ref{rem:kernel}). We also introduce the stream function $\psi = G \ast \omega$, where 
\[
G(x) = \frac{1}{2\pi} \log(|x|),
\]
denotes the fundamental solution of the Laplacian. Note that with this choice of $G$, we have $\Delta \psi = \omega$.

\begin{lemma} \label{lem:structintermediate}
If $\omega = \curl(u) \in L^1_x\cap L^\infty_x$ has compact support, $u = K\ast \omega$, $\psi = G \ast \omega$, then
\begin{gather} \label{eq:structintermediate}
\begin{aligned}
\frac 12 \fint_{B_r(0)} &\int_{\dom} |u(x+h) - u(x)|^2 \, dx \, dh
&=
-\int_{\dom} \left(\psi(x)-[\psi]_r(x)\right) \omega(x) \, dx,
\end{aligned}
\end{gather}
where $[\psi]_r(x) := \fint_{B_r(0)} \psi(x+h) \, dh$.
\end{lemma}

\begin{proof}
Note that $\psi \in L^\infty_{x,\loc}$, $\nabla^\perp \psi = u \in L^2_{x,\loc}$, $\Delta \psi = \omega \in L^1_x\cap L^\infty_x \subset L^2_x$, by the assumptions of this lemma. 
Choose $R_0>0$, such that $\supp(\omega) \subset B_{R_0}(0)$. Fix $h\in \dom$ for the moment. Denote $U(x) := u(x+h)-u(x)$, $\Psi(x) := \psi(x+h)-\psi(x)$, $\Omega(x) := \omega(x+h)-\omega(x)$. Note that $\Omega$ is compactly supported, and $\int_\dom \Omega(x) \, dx = 0$. In particular, following Remark \ref{rem:kernel}, this implies that $U = K\ast \Omega \in L^2_x$. Furthermore, $U$ and $\Psi$ are smooth on $\dom\setminus \overline{B_{R_0}}(0)$, and 
\[
\Psi(x) \lesssim |\log|x||, \qquad   |\nabla \Psi(x)|= |U(x)|\lesssim |x|^{-2},
\]
as $|x|\to \infty$.  Let now $R>R_0$. After an integration by parts, we find
\begin{align*}
\int_{B_R(0)}  |U|^2\, dx
&=
\int_{B_R(0)} |\nabla \Psi|^2 \, dx
\\
&=
\underbrace{\int_{\partial B_R(0)} \Psi (\nabla\Psi\cdot \nu) \, d\sigma}_{\lesssim R^{-1} \log(R)}
 -\int_{B_R(0)}  \Psi(x)\Omega(x) \, dx,
\end{align*}
where $\nu$ denotes the outward pointing normal vector to $\partial B_R(0)$. Letting $R\to \infty$, and replacing $U,\Omega,\Psi$ by their definitions, we can now write
\begin{align*}
\int_{\dom}  |u(x+h)-u(x)|^2\, dx
&=
-\int_{\dom}  [\psi(x+h)-\psi(x)][\omega(x+h)-\omega(x)] \, dx
\\
&=
-\int_{\dom}  [\psi(x+h)\omega(x+h) + \psi(x)\omega(x)] \, dx 
\\
&\qquad 
+ \int_{\dom}  [\psi(x+h)\omega(x) + \psi(x)\omega(x+h)] \, dx
\\
&\explain={\text{(change of variables)}}
-\int_{\dom}  2[\psi(x)\omega(x)] \, dx 
\\
&\qquad 
+ \int_{\dom}  [\psi(x+h)\omega(x) + \psi(x-h)\omega(x)] \, dx
\\
&= 
-\int_{\dom}  [\psi(x)-\psi(x+h)]\omega(x) \, dx
\\
&\qquad 
-\int_{\dom}  [\psi(x)-\psi(x+h)]\omega(x) \, dx.
\end{align*}
We point out that the manipulations on the right-hand side are justified, since $\psi \in L^\infty_{x,\loc}$ and $\omega \in L^\infty_x$ with $\supp(\omega)$ compact. It follows that, upon integration over $h\in B_r(0)$, the two last terms can be combined and 
\begin{align}
\fint_{B_r(0)} \int_{\dom}  |u(x+h)-u(x)|^2\, dx \, dh
&=
-2\int_{\dom}  [\psi(x)-[\psi]_r(x)]\omega(x) \, dx,
\end{align}
where $[\psi]_r(x) := \fint_{B_r(0)} \psi(x+h) \, dh$. The claim follows.
\end{proof}

We next wish to express the difference $\psi(x) - [\psi]_r(x)$ in terms of the vorticity. For this we will need the following lemma:

\begin{lemma} \label{lem:log}
Let $e \in S^1$ be a unit vector, and let $s\ge 0$. Then
\[
\fint_{S^1} \log |e + s\sigma| \, d\sigma = \log(s)^{+},
\]
where $\log(s)^+ := \max(0,\log(s))$ denotes the positive part of the logarithm.
\end{lemma}
\begin{proof}
(i) We note that $z\mapsto \log(|e+z|)$ is harmonic in the ball $z\in B_1(0)$. From the mean value property of harmonic functions, it therefore follows that, for $0\le s < 1$:
\[
\fint_{S^1} \log |e + s\sigma| \, d\sigma = \log |e+0| = 0 = \log(s)^+.
\]
 (ii) Now assume that $s> 1$. We introduce a polar angle $\theta\in [0,2\pi)$ such that $\sigma \in S^1$ has the parametrization $\sigma = (\cos(\theta),\sin(\theta))$, and $e = (1,0)$. We can write
\begin{align*}
\log| e + s\sigma | 
&= \frac 12 \log \left( 1 + 2s\cos(\theta) + s^2 \right)
\\
&= \frac 12 \log\left(s^2\right) + \frac 12 \left( s^{-2} + 2s^{-1}\cos(\theta) + 1 \right)
\\
&= \log(s) + \log |e + s^{-1}\sigma |.
\end{align*}
Thus, in this case we find, using also (i), that
\[
\fint_{S^1} \log |e + s\sigma| \, d\sigma 
=
\fint_{S^1} \left(\log(s) + \log |e + s^{-1}\sigma | \right) \, d\sigma 
=
\log(s) = \log(s)^+,
\]
where the integral over the second integrand vanishes, since $s^{-1} < 1$. The claimed equality for $s=1$ is obtained by taking the limit $s\to 1$.
\end{proof}

\begin{lemma} \label{lem:logrewrite}
Let $r>0$, $z\in \dom\setminus \{0\}$. Then
\[
\fint_{B_r(0)} [\log |z+h|-\log|z|] \, dh
= 
2\pi \Sigma\left(\frac{|z|}{r}\right),
\]
where 
\[
\Sigma(\rho) := 
\begin{cases}
\frac{1}{4\pi}\left(|\log(\rho^2)| - 1 + \rho^2\right), & (\rho \le 1), \\
0, & (\rho > 1).
\end{cases}
\]
\end{lemma}

\begin{proof}
We have
\begin{align*}
I &:= 
\fint_{B_r(0)}   
 \left[\log(|z+h|) - \log(|z|)\right]\, dh
\\
&= 
 \frac{1}{\pi r^2} \int_{0}^r \rho \int_{S^1}  \log\left(\frac{|z+\rho \sigma|}{|z|}\right) \, d\sigma \, d\rho 
 \\
 &=
  \frac{2}{r^2} \int_{0}^r \rho \fint_{S^1} \log\left(\left|\frac{z}{|z|}+\frac{\rho}{|z|} \sigma\right|\right) \, d\sigma \, d\rho 
\end{align*}
By Lemma \ref{lem:log}, 
\[
\fint_{S^1} \log\left(\left|\frac{z}{|z|}+\frac{\rho}{|z|} \sigma\right|\right) \, d\sigma
= \log\left(\frac{\rho}{|z|}\right)^+.
\]
Hence, we find that
\[
I = \frac{2}{r^2} \int_{0}^r  \rho \log\left(\frac{\rho}{|z|}\right)1_{[\rho \ge |z|]}  \, d\rho.
\]
For $|z|\ge r$, we clearly have $I=0 = \Sigma(|z|/r)$. If $|z|< r$, then 
\begin{align*}
I &= \frac{2}{r^2} \int_{|z|}^r  \rho \log\left(\frac{\rho}{|z|}\right)  \, d\rho
\\
&= \frac{2|z|^2}{r^2} \int_{1}^{r/|z|} s \log(s) \, ds 
\\
&= \frac{2|z|^2}{r^2} \left[\frac{1}{4}s^2(\log(s^2) - 1)\right]_{1}^{r/|z|}
\\
&= \frac{1}{2} \left(\log((r/|z|)^2) - 1 + (|z|/r)^2\right)
\\
&= 2\pi\Sigma(|z|/r),
\end{align*}
where
\[
\Sigma(\rho) 
= 
\frac{1}{4\pi} \left(|\log(\rho^2)| - 1 + \rho^2\right).
\]
\end{proof}

\begin{lemma} \label{lem:psipsiav}
Let $\psi = G\ast \omega$, $\omega \in L^1_x\cap L^\infty_x$, $\supp(\omega)$ compact. Then 
\[
\psi(x) - [\psi]_r(x)
=
-\int_{B_r(0)} \Sigma\left(\frac{|h|}{r}\right)\omega(x+h)\, dh,
\]
where $[\psi]_r(x) := \fint_{B_r(0)} \psi(x+h) \, dh$.
\end{lemma}

\begin{proof}
We have
\[
 \psi(x) - [\psi]_r(x)
= 
\frac{-1}{2\pi}
\int_\dom 
\left(
\fint_{B_r(0)} \left[\log |z+h| - \log |z|\right] \, dh 
\right)
\, \omega(x+z) \, dz.
\]
By Lemma \ref{lem:logrewrite}, the average on the right-hand side can be simplified to yield
\[
\psi(x) - [\psi]_r(x) 
=
-\int_\dom \Sigma\left(\frac{|z|}{r}\right) \, \omega(x+z) \, dz.
\]
\end{proof}

\begin{lemma} \label{lem:psimax}
For any $\omega \in L^1_x \cap L^\infty_x$, $\supp(\omega)$ compact, and $\psi = G \ast \omega$, we have
\[
\Vert \psi - [\psi]_r \Vert_{L^\infty} \le \int_0^r \frac{M_s(\omega) \, ds}{s}, \quad \forall r>0.
\]
We have denoted by $[\psi]_r$ the local average of $\psi$ over a ball of radius $r$, i.e. $[\psi]_r(x) := \fint_{B_r(0)} \psi(x+h) \, dh$.
\end{lemma}

\begin{proof}
Let $x\in \mathbb{R}^2$ be arbitrary. By Lemma \ref{lem:psipsiav}, and the fact that $\Sigma(\rho) \le |\log(\rho)|$, we have
\begin{align*}
|\psi(x)-[\psi]_r(x)|
\le 
\int_{B_r(0)} \left|\log\left(\frac{|h|}{r}\right)\right||\omega(x+h)|\, dh.
\end{align*}
Writing the last integral in polar coordinates, we obtain
\begin{align*} 
|\psi(x)-[\psi]_r(x)|
&\le
\int_0^r \left|\log\left(\frac{s}{r}\right)\right| \left(s \int_{S^1} |\omega(x+s\sigma)|\, d\sigma\right)\, ds
\end{align*}
Next, we note that 
\[
s \int_{S^1} |\omega(x+s\sigma)|\, d\sigma = \frac{d}{ds} \int_{B_s(x)} |\omega(y)| \, dy.
\]
Let $m(x;s) := \int_{B_s(x)} |\omega(y)| \, dy$, so that 
\begin{align}
|\psi(x)-[\psi]_r(x)|
\le
\int_0^r \left|\log\left(\frac{s}{r}\right)\right| \frac{dm(x;s)}{ds}\, ds.
\end{align}
Integration by parts yields
\begin{align*}
|\psi(x)-[\psi]_r(x)|
&\le
\int_0^r \frac{m(x;s)}{s} \, ds
\le 
\int_0^r \frac{M_s(\omega)}{s} \, ds.
\end{align*}
The claimed estimate follows.
\end{proof}


\section{Strong convergence of vortex-blob method} \label{app:vortexblob}

\if{
\begin{lemma} \label{lem:vortdecaysmooth}
Let $0 < \epsilon \le 1$. If there exist $\overline{M} > 0$ and $0 < \alpha \le 2$, such that 
\[
M_r(\omega_\epsilon) \le \overline{M} r^\alpha, \quad \text{for all $r \ge \epsilon$},
\]
uniformly as $\epsilon \to 0$, then 
\[
M_r(\omega^\epsilon) \le 
\begin{cases}
C r^{\lambda}, \quad (r< \epsilon) \\
C r^{\alpha},  \quad (r\ge \epsilon)
\end{cases}
\]
where $C = C(\Vert \omega_0 \Vert_{\mathcal{M}},\overline{M},\Vert \phi \Vert_{L^\infty})$ and 
\[
\lambda := \frac{2\alpha}{2+\alpha}.
\]
\end{lemma}

\begin{proof}
To prove the conclusion of this lemma, we note that 
\begin{align*}
M_r(\omega^\epsilon) 
&= \sup_x \sup_{t\in [0,T]} \int_{B_r(x)} |\omega^\epsilon(z,t)| \, dz
\end{align*}
and
\begin{align*}
\int_{B_r(x)} |\omega^\epsilon(z,t)| \, dz
&\le \int_{B_r(x)} \sum_{j} |\xi_j| \phi_\epsilon(x_j-z) \, dz
\\
&\le 
\frac{\Vert \phi\Vert_{L^\infty}}{\epsilon^2}\sum_j |\xi_j| |B_r(x) \cap B_\epsilon(x_j)|.
\end{align*}
We obtain two bounds from the latter expression: The first bound is useful when $r\lesssim \epsilon$, and is obtained from the trivial estimate
\[
|B_r(x) \cap B_\epsilon(x_j)| \le |B_r(x)| = \pi r^2,
\]
which yields
\begin{align}\label{eq:VM1}
M_r(\omega^\epsilon) 
\le \pi\Vert \phi\Vert_{L^\infty} \Big(\sum_{j} |\xi_j| \Big) \left(\frac{r}{\epsilon}\right)^2 
\le \pi\Vert \phi\Vert_{L^\infty} \Vert \omega_0 \Vert_{\mathcal{M}} \left(\frac{r}{\epsilon}\right)^2.
\end{align}
For the second upper bound, we utilize instead
\[
\sum_{j} |\xi_j| |B_r(x) \cap B_\epsilon(x_j)| 
\le 
\sum_{|x_j - x|\le r+\epsilon} |\xi_j| |B_{\epsilon}(x)| = \pi \epsilon^2 \int_{B_{r+\epsilon}(x)} |\omega_\epsilon(z,t)| \, dz,
\]
and the assumption $\int_{B_{r+\epsilon}(x)} |\omega_\epsilon(z,t)| \, dz\le M_{r+\epsilon}(\omega_\epsilon) \le \overline{M} (r+\epsilon)^\alpha$, to obtain
\begin{align}\label{eq:VM2}
M_r(\omega^\epsilon) \le \pi \Vert \phi\Vert_{L^\infty} \overline{M} (r+\epsilon)^\alpha.
\end{align}
Combining \eqref{eq:VM1} and \eqref{eq:VM2}, we thus have
\begin{align} \label{eq:VM3}
M_r(\omega^\epsilon) \le 
\begin{cases}
C\left(\dfrac{r}{\epsilon}\right)^2, & {(I)}\\
C(r+\epsilon)^\alpha, & {(II)}
\end{cases}
\end{align}
for all $r\ge 0$, where $C = C(\Vert \omega_0 \Vert_{\mathcal{M}},\overline{M},\Vert \phi\Vert_{L^\infty})$ is independent of $\epsilon > 0$. 

Fix $\epsilon > 0$ for the moment. For $r\le \epsilon^{1+\epsilon}$, we clearly have from \eqref{eq:VM3}(I):
\[
M_r(\omega^\epsilon) 
\le C \left(\frac{r}{\epsilon}\right)^2
= C \Big(\underbrace{\frac{r}{\epsilon^{1+\epsilon}}}_{\le 1} \, r^{\epsilon}\Big)^{2/(1+\epsilon)}
\le C r^{2\epsilon/(1+2\epsilon)}.
\]
On the other hand, for $1\ge r \ge \epsilon^{1+\epsilon}$, we have $\max(\epsilon,r) \le r^{1/(1+\epsilon)}$, and hence by \eqref{eq:VM3}(II):
\[
M_r(\omega^\epsilon)
\le C(r+\epsilon)^\alpha \le C(2r^{1/(1+\epsilon)})^\alpha \le 2^\alpha C r^{\alpha/(1+\epsilon)} \le 4 C r^{\alpha/(1+\epsilon)}.
\]
Setting the free parameter $\epsilon := \alpha/2$, we find 
\[
M_r(\omega^\epsilon) \le 4C r^{2\alpha/(2+\alpha)},
\]
for all $0 \le r \le 1$, where we can in fact explicitly identify $C = \pi \Vert \phi \Vert_{L^\infty} \max(\overline{M},\Vert \omega_0 \Vert_{\mathcal{M}})$.

Finally, we note that for $r\ge \epsilon$ (including the case $r\ge 1$), this estimate can be further improved, since in this case we obtain
\[
M_r(\omega^\epsilon) \le C (r+\epsilon)^\alpha \le 2^\alpha C r^\alpha \le 4C r^\alpha,
\]
from \eqref{eq:VM3}(II). Thus, $M_r(\omega^\epsilon)$ satisfies the claimed estimate with a constant depending only on $\Vert \omega_0 \Vert_{M}$, $\overline{M}$ and $\Vert \phi \Vert_{L^\infty}$.
\end{proof}
}
\fi

\begin{proof}[Proof of \ref{thm:strongconv}]
We recall that, by assumption, the vorticity remains compactly supported for all $t\in [0,T]$. It is easy to see that $\int \omega^\epsilon(t) \, dx = 0$ for all $t\in [0,T]$. As a consequence, no distinction needs to be made between convergence of $u^\epsilon$ in $L^2_{x,\loc}$ and $L^2_x$. The weak convergence $u^\epsilon \weaklyto u$ in $L^2_tL^2_{x}$ to a weak solution of the incompressible Euler equations follows directly from the work \cite{LiuXin1995}. To see that the convergence $u^\epsilon \to u$ is in fact strong, we recall that by Lemma \ref{lem:vortmoll} and the assumptions on the uniform algebraic decay of $M_r(\omega_\epsilon)$ in this theorem, we have
\[
M_r(\omega^\epsilon) \le M_r(\omega_\epsilon) \le C |\log(r)^\beta, \quad \forall r > 0,
\]
for $C>0$, $\beta >1$.
By Corollary \ref{cor:algebraicdecay}, this implies that
\begin{align*}
\sup_{t\in [0,T]}
\fint_{B_r(0)} \int_{\R^2} |u^\epsilon(x+h)-u^\epsilon(x)|^2 \, dx \, dh
&\le 
\frac{C}{\beta-1} \Vert \omega^\epsilon \Vert_{\mathcal{M}} |\log(r)|^{\beta-1}
\\
&\le
\frac{C}{\beta-1} \Vert \omega_0 \Vert_{\mathcal{M}} |\log(r)|^{\beta-1}.
\end{align*}
The strong relative compactness of $u^\epsilon$ as $\epsilon \to 0$ now follows from Proposition \ref{prop:compactness}. Since $u^\epsilon$ is relatively compact in $L^2_t L^2_{x}$ and converges weakly $u^\epsilon \weaklyto u$, we conclude that in fact $u^\epsilon \to u$ \emph{strongly} in $L^2_t L^2_{x}$. 

To see that $u$ is energy conservative, we note that 
\begin{align*} 
\int_0^T \left| \Vert u^\epsilon(t) \Vert_{L^2_x}^2 - \Vert u(t) \Vert_{L^2_x}^2 \right| \, dt
&=
\int_0^T \left| \left\langle u^\epsilon(t) - u(t), u^\epsilon(t) + u(t) \right\rangle_{L^2_x} \right| \, dt
\\
&\le 
\int_0^T \left\Vert u^\epsilon(t) - u(t) \right\Vert_{L^2_x}  \left\Vert u^\epsilon(t) + u(t) \right\Vert_{L^2_x} \, dt
\\
&\le 
\left(\int_0^T \left\Vert u^\epsilon(t) - u(t) \right\Vert_{L^2_x}^2\, dt\right)^{1/2}
\left(\int_0^T \left\Vert u^\epsilon(t) + u(t) \right\Vert_{L^2_x}^2\, dt\right)^{1/2}
\\
&\le \Vert u^\epsilon - u \Vert_{L^2_{t,x}} \left(\Vert u^\epsilon \Vert_{L^2_{t,x}} + \Vert u \Vert_{L^2_{t,x}}\right).
\end{align*}
From the strong convergence $u^\epsilon \to u$ in $L^2_{t,x}$, it follows that
\[
\Vert u^\epsilon - u \Vert_{L^2_{t,x}} \to 0, \quad (\epsilon \to 0),
\]
and that there exists a constant $C>0$, such that
\[
 \Vert u^\epsilon \Vert_{L^2_{t,x}} + \Vert u \Vert_{L^2_{t,x}} 
 < C, \quad \forall \, \epsilon > 0.
 \]
Thus, $t \mapsto \Vert u^\epsilon(t) \Vert_{L^2_x}^2$ converges to $t\mapsto \Vert u(t) \Vert_{L^2_x}^2$ in $L^1([0,T])$. From the strong convergence in $L^1([0,T])$, it follows that we can extract a subsequence $\epsilon_k \to 0$, such that
\[
\Vert u^{\epsilon_k}(t) \Vert_{L^2_x}^2 \to \Vert u(t) \Vert_{L^2_x}^2, 
\quad \text{for almost every $t\in [0,T]$.}
\]
On the other hand, by Proposition \ref{prop:energyerror}, the algebraic bound on the vorticity maximal function implies that for any $t\in [0,T]$, we have $\Vert u^\epsilon(t)\Vert_{L^2_x}^2 \to \Vert u_0 \Vert_{L^2_x}^2$ as $\epsilon \to 0$. Hence
\[
\Vert u_0 \Vert_{L^2_x} 
=
\lim_{k\to \infty} \Vert u^{\epsilon_k}(t) \Vert_{L^2_x}
=
\Vert u(t) \Vert_{L^2_x},
\]
for almost all $t\in [0,T]$.
\end{proof}

\begin{lemma}\label{lem:Ealong}
Let 
\[
E^\epsilon(t) := \int |u^\epsilon(t)|^2 \, dx = \int \psi^\epsilon(t) \omega^\epsilon(t) \, dx,
\quad
\mathcal{E}^\epsilon(t) := \int \psi^\epsilon(t) \omega_\epsilon(t) \, dx,
\]
so that $\mathcal{E}^\epsilon(t)$ is the invariant corresponding to the energy of the discretized vortex system, i.e.
\[
\mathcal{E}^\epsilon(t) \equiv \sum_{i,j} G_\epsilon(|X_i(t) - X_j(t)|) \xi_i \xi_j,
\]
where $G_\epsilon(|x|) := -\frac{1}{4\pi}\log(|x|^2 + \epsilon^2)$.
If there exist $C_M,\alpha>0$, such that $M_r(\omega^\epsilon) \le C_M r^\alpha$,  then 
\[
\left|
E^\epsilon(t) - \mathcal{E}^\epsilon(t)
\right|
\le 
C \epsilon^\alpha,
\]
where $C = C(\Vert \omega_0 \Vert_{\mathcal{M}}, C_M, \alpha, \phi)$.
\end{lemma}

\begin{proof}
Let $\psi^\epsilon := G_\epsilon \ast \omega_\epsilon$.
We have
\begin{align*}
|E^\epsilon(t) - \mathcal{E}^\epsilon(t)|
&= 
\left| 
\int_D \psi^\epsilon(x,t) \left[\omega^\epsilon(x,t) - \omega_\epsilon(x,t)\right] \, dx
\right|
\\
&=
\left| 
\sum_{j} \xi_j(t)
\left[(\psi^\epsilon \ast \phi_\epsilon)(X_j,t) - \psi^\epsilon(X_j,t)\right]
\right|
\\
&\le \left(\sum_j |\xi_j(t)|\right) \left\Vert (\psi^\epsilon \ast \phi_\epsilon)(\slot,t) - \psi^\epsilon(\slot,t) \right\Vert_{L^\infty_x}.
\end{align*}
By Lemma \ref{lem:psimax}, $M_r(\omega^\epsilon) \le C_Mr^\alpha$ implies
\[
\left\Vert (\psi^\epsilon \ast \phi_\epsilon)(\slot,t) - \psi^\epsilon(\slot,t) \right\Vert_{L^\infty}
\le 
\tilde{C}\epsilon^\alpha,
\]
where $\tilde{C} = \tilde{C}(C_M,\alpha,\phi)$. We furthermore note that 
\[
\sum_{j} |\xi_j| \le \Vert \omega_0 \Vert_{\mathcal{M}}.
\]
Hence $|E^\epsilon(t) - \mathcal{E}^\epsilon(t)| \le C\epsilon^\alpha$ for $C = \Vert \omega_0 \Vert_{\mathcal{M}} \tilde{C}$ depending only on $C_M$, $\alpha$, $\phi$, $\Vert \omega_0\Vert_{\mathcal{M}}$, as claimed.
\end{proof}

\begin{proposition} \label{prop:energyerror}
Let $\omega_0 \in M\cap H^{-1}$ be initial vorticity data for the incompressible Euler equations. Let $\{\omega_\epsilon\}_{\epsilon > 0}$ be a good vortex-blob approximation. If there exist constant $C,\alpha>0$, such that $M_r(\omega^\epsilon) \le Cr^\alpha$ uniformly for all $\epsilon > 0$, then $E^\epsilon(t) \to E_0$ as $\epsilon \to 0$, for all $t \in [0,T]$.
\end{proposition}

\begin{proof}
By the definition of a good vortex-blob approximation, we have $E^\epsilon(t=0) \to E_0$ as $\epsilon \to 0$ at the initial time. By Lemma \ref{lem:Ealong}, we have $E^\epsilon(t) - \mathcal{E}^\epsilon(t) \to 0$ as $\epsilon \to 0$. Furthermore, for any $\epsilon > 0$ fixed, we have $\mathcal{E}^\epsilon(t) = \mathcal{E}^\epsilon(0)$ due to the Hamiltonian character of the vortex-blob approximation. Thus,
\begin{align*}
E^\epsilon(t)
&= \left[E^\epsilon(t) - \mathcal{E}^\epsilon(t)\right] + \mathcal{E}^\epsilon(t)
\\
&= \left[E^\epsilon(t) - \mathcal{E}^\epsilon(t)\right] + \mathcal{E}^\epsilon(0)
\\
&= \underbrace{\left[E^\epsilon(t) - \mathcal{E}^\epsilon(t)\right]}_{\to 0} + \underbrace{\left[\mathcal{E}^\epsilon(0)-E^\epsilon(0)\right]}_{\to 0} + \underbrace{E^\epsilon(0)}_{\to E_0},
\end{align*}
converges to $E_0$ as $\epsilon \to 0$.
\end{proof}

\bibliographystyle{unsrt}
\bibliography{concentration-vortex-sheet}

\end{document}